\documentclass[11pt,reqno]{amsart}
\usepackage{amsmath,amsthm,amssymb,latexsym,color}
\usepackage{hyperref}
\usepackage{lipsum}
\usepackage{xcolor}
\date{}

\usepackage{graphicx}
\newtheorem{theorem}{Theorem}[section]
\newtheorem*{theorem*}{Theorem}

\newtheorem{lemma}[theorem]{Lemma}

\newtheorem{proposition}[theorem]{Proposition}

\theoremstyle{definition}
\newtheorem{remark}[theorem]{Remark}

\theoremstyle{plain}

 \usepackage{graphicx}

\newcommand{\N}{\mathbb{N}}

\newcommand{\R}{\mathbb{R}}
\newcommand{\PSD}{\mathcal{H}_d^+}
\newcommand{\sym}{\mathcal{H}_d}
\newcommand{\supp}{{\rm Supp\, }}
\newcommand{\E}{\mathbb{E}}
\newcommand{\gen}{\mathcal{L}}
\newcommand{\Var}{{\rm Var}}
\newcommand{\Tr}{{\rm Tr}}
\newcommand{\Dir}{\mathcal{E}}
\newcommand{\LL}{\rm{L}^2(\mu)}
\newcommand{\DD}{\mathcal{D}(\mathcal{L})}

\def\Prob{{\mathbb P}}

\newcommand{\w}{\widetilde}

\newcommand{\efface}[1]{}

\begin{document}

\author{Richard Aoun}
\author{Marwa Banna}
\author{Pierre Youssef}
\title{Matrix Poincar\'e inequalities and concentration}

\maketitle

\begin{abstract}
We  show that any probability measure satisfying a Matrix Poincar\'e inequality 
with respect to some reversible Markov generator satisfies an exponential matrix concentration inequality  depending on the associated matrix carr\'e du champ operator.   This extends to the matrix setting a  classical phenomenon in the scalar case. Moreover, the proof gives rise to new matrix trace inequalities which could be of independent interest. We then apply this general fact by establishing   matrix Poincar\'{e} inequalities to  derive matrix concentration inequalities   for   Gaussian measures, product measures and  for   Strong Rayleigh measures.  The latter      represents the first instance of matrix concentration for general matrix functions of negatively dependent random variables. 
\end{abstract}

\section{Introduction and main results}

Concentration inequalities are versatile tools which found use in several pure and applied mathematical problems. 
While in their essence, these inequalities are just a quantification of the law of large numbers, they represent further illustrations of deep high dimensional phenomena in areas across Mathematics. 
On a conceptual level, they assert that a random variable measurable with respect to a large number of independent (or weakly dependent) random variables and ``depending little'' on each individually, is almost constant with high probability. While many methods were developed to prove concentration inequalities, perhaps the most insightful one is the one based on functional inequalities such as Poincar\'e and log-Sobolev inequalities. Indeed, these functional inequalities serve as a further illustration of the conceptual description we mentioned above. For instance, in its classical form, Poincar\'e inequality relates the variance of a function of a random variable to the average length of the gradient of the function. One then readily sees that if a function varies little locally (in the usual  sense of variations), then with high probability it also varies little when evaluated at a random point. More generally, these inequalities relate statistics of the measure, such as the variance and entropy of any function, to the derivative along a semi-group associated with a Markov process generating the measure. Beside the concentration phenomenon, such functional inequalities provide a further understanding of the measure as they are intimately connected to the convergence rate of the corresponding Markov process generating it.  The interconnection between concentration and functional inequalities is by now very well understood, and the use of such inequalities to derive concentration proved to be very powerful due to its flexibility in dealing with any measure, provided one can architect a suitable Markov process which generates it. We refer to \cite{BLM-book} and \cite{ledoux} for more on classical concentration inequalities.

Matrix concentration inequalities are noncommutative extensions of their scalar counterpart and have been extensively developed in the last decade \cite{AW, Oliveira, PMT-EffronStein, Tropp-martingale, Tropp-con}.  A big effort was made to transfer our understanding of the scalar case to the matrix one. In this direction, many papers were devoted to extending scalar methods for deriving concentration inequalities in the matrix setting.  We refer to the book of Tropp \cite{Tropp-survey} for a detailed introduction to the subject and an extensive list of references. 
As mentioned in the previous paragraph, the approach leading to concentration inequalities based on functional inequalities has been successful in the scalar case: it allowed to establish concentration inequalities in dependent settings and treat general functions beyond the example of sums of random variables. 
In light of this, several efforts were made to extend this theory to the matrix setting.  
Several papers were devoted to properly defining the matrix entropy and establishing its basic properties such as the subadditivity \cite{Carlen, Chen-Tropp, Ch-Hs, Ch-Hs-new, Ch-Hs-17}. In \cite{Chen-Tropp}, Chen and Tropp aimed at extending $\phi$-Sobolev inequalities to the matrix setting. 
In \cite{Ch-Hs, Ch-Hs-new}, Cheng and Hsieh further investigated the notion of $\Phi$-entropy for operator valued functions and established several matrix functional inequalities such as a matrix analogue of the Efron-Stein inequality. 
The subadditivity of matrix entropy was used in \cite{Chen-Tropp} with the aim of developing a matrix version of the entropy method. 
However, as was noted in \cite{Chen-Tropp}, the attempt to adapt the Herbst's argument had some shortcomings requiring additional assumptions to be successfully implemented. In addition to \cite{Chen-Tropp}, we should mention the paper \cite{PMT-EffronStein} where a family of exponential matrix Efron-Stein inequalities are established and turned into matrix concentration. However,  the literature lacks a unified general framework for relating  matrix functional inequalities to matrix concentration.

In this paper, we focus on Poincar\'e inequalities and aim to implement a general procedure turning a {\it matrix} Poincar\'e inequality into a concentration inequality. In the scalar case, such a procedure  was first introduced by Gromov and Milman  \cite{Gro-Mil}
  and alternative arguments were later developed. We will adapt here the approach of Aida and Stroock \cite{AS} (see also  \cite[Section 2.5]{Ledoux-lecturenotes}). One of our contributions is to elaborate such a procedure, then  establish matrix Poincar\'e  inequalities and use this  to derive new matrix concentration inequalities for Gaussian measures and negatively dependent measures.

\subsection*{\bf Matrix Poincar\'e implies matrix concentration}
Let us denote by $\sym$ (resp.~$\PSD$) the set of $d\times d$ Hermitian (resp.~positive semi-definite)  matrices. Given a probability measure $\mu$ on some Polish space $\Omega$ and $f:\Omega\to \sym$ whose matrix coefficients belong to $\LL$, the variance of $f$ is given by 
$$
\Var_\mu(f)= \E_\mu[f^2]-(\E_\mu f)^2, 
$$
where $\E_\mu:=\int f d\mu$. 
It can be easily checked that $\Var_\mu(f)\succeq 0$, where $\succeq$ refers to the positive semi-definite ordering. 
We will say that $\mu$ satisfies a \textit{matrix Poincar\'e inequality} with constant $\alpha$ and \textit{matrix Markov generator} $\gen$ with Dirichlet domain $\DD$ if for any $f\in \DD$ we have 
$$
\Var_\mu(f) \preceq \alpha \Dir(f),
$$
where $\Dir(f)=-\E_\mu[f\gen f]$ is the \textit{matrix Dirichlet form} associated with $\gen$. As is verified in Proposition \ref{prop: dirichlet properties}, this definition makes sense as $\Dir(f)\in \PSD$.   The notion of matrix Poincar\'e  inequality (with respect to the positive semi-definite ordering) 
appears in the works of Chen and Hsieh \cite{Ch-Hs, Ch-Hs-new} although not expressed in the language of semigroups. Together with Tomamichel \cite{Ch-Hs-17}, the aforementioned authors later developed the basic definitions and properties of semigroups acting on matrix functions, as well as the matrix Dirichlet form and matrix carr\'e du champ operator, which we use here. These notions will be recalled in Section \ref{sec: matrix generator} for completeness, and can be thought of at this stage as natural extensions of their scalar counterpart by considering the action of the Markov generator on each entry of the matrix valued function. Let $\Gamma$ be the matrix carr\'{e} du champ operator associated to the matrix generator $\gen$ defined on an algebra $\mathcal{A}$ of $\DD$.  Our first main result states as follows.

\begin{theorem}\label{th: poincare to concentration}
Let $\mu$ be a probability measure on some Polish space $\Omega$. Suppose that $\mu$ satisfies a matrix Poincar\'e inequality with constant $\alpha$ and matrix Markov generator $\gen$ reversible with respect to $\mu$.  Then  for any  $f\in  \mathcal{A}$ and       $t\geq 0$, we have 
$$
\mu\Big( \lambda_{max}\big(f-\E_\mu f\big)\geq t\Big)\leq 
d\exp\bigg(-\frac{t^2}{2\alpha v_f+ t\sqrt{2\alpha v_f}}\bigg),
$$
where $v_f= \big\Vert \Vert \Gamma(f)\Vert\big\Vert_{L_\infty}$.
\end{theorem}

 The above theorem provides a  general machinery turning a matrix Poincar\'e inequality into a corresponding matrix concentration inequality.  
In view of such general phenomenon paralleling its scalar counterpart, establishing a matrix concentration inequality is reduced to proving a matrix Poincar\'e inequality. To this aim, for a given probability measure, the main task lies in designing the appropriate Markov generator and calculating the corresponding matrix carr\'e du champ operator.

The proof of Theorem~\ref{th: poincare to concentration} hides many challenging obstacles arising because of noncommutativity. As is customary,  obtaining a concentration inequality follows by combining a Chernoff bound with an estimate on the Laplace transform. 
The Poincar\'e inequality is then used to obtain a recursive relation involving the Laplace transform, which when properly arranged produces the desired bound on the Laplace transform. This simple looking procedure carries a great amount of difficulties when one attempts to extend it to the matrix setting. For instance, given $g:\Omega\to \sym$, one starts applying the matrix Poincar\'e inequality to $e^{g}$ to get 
$$
\E_\mu[e^{2g}]\preceq (\E_\mu e^g)^2 + \alpha \Dir(e^g). 
$$
In the scalar case, $\Dir(e^g)$ can be easily related to the Laplace transform which would automatically translates the above relation into a recursive formula on the Laplace transform. Such a relation is far from trivial in the matrix setting and requires the development of new matrix trace inequalities which could be of independent interest. 
Such a relation is established in Section~\ref{sec: matrix trace ineq} (see Theorem~\ref{th: dir-laplace}) where, in particular, a new matrix trace inequality is elaborated (see Theorem~\ref{th: mean value}). While obtaining a recursive formula is the end of the story in the scalar case, such a relation cannot be directly iterated in the matrix setting. Indeed, since the square function is not operator monotone, one cannot reapply the same procedure to bound $(\E e^g)^2$ in the above formula. 
To overcome this issue, we exploit the operator monotonicity of the trace of such functions and combine it with special convexity arguments to implement the iterative procedure.

\subsection*{\bf Matrix Poincar\'e and Concentration for product measures}
We derive a matrix Poincar\'e inequality for the standard Gaussian measure and use the mechanism in Theorem~\ref{th: poincare to concentration} to deduce corresponding concentration.  
To this aim, we consider the Ornstein-Uhlenbeck Markov process whose matrix Dirichlet form is precisely the expectation of the sum of the squared partial derivatives matrices.  In this case, we obtain the following matrix Gaussian Poincar\'e inequality. 

\begin{theorem}\label{th: Gaussian-Poincare}
Let $\mu= \mu_1 \otimes \ldots \otimes \mu_n$  be the standard  Gaussian measure on $\R^n$. Let $f: \R^n\to \sym$ be such that all of its matrix coefficients, together with their partial derivatives,  are smooth and in $\LL$. Then 
\[
\Var_\mu (f) \preceq  \int \sum_{i=1}^n (\partial_i f)^2 d\mu,
\]
where $\partial_i f (x_1, \ldots, x_n):= \frac{\partial f }{\partial x_i}  (x_1, \ldots, x_n)$ is the matrix whose entries are the $i$-th partial derivatives of the corresponding entries of $f$. 
\end{theorem}

 The above inequality appears to be new and extends to the matrix setting the scalar Gaussian Poincar\'e inequality. Indeed, when $f$ is a scalar function, the right-hand side is precisely the integral of the Euclidean norm squared of the gradient of $f$. 
A related Poincar\'e inequality for the Gaussian unitary ensemble was obtained in \cite[Theorem~4]{Ch-Hs-new} with the trace applied to both sides of the inequality. 
 Combined with Theorem~\ref{th: poincare to concentration}, the above statement implies the following concentration inequality. 

\begin{theorem}\label{thm:Gaussian-prod-Poin}
Let $\mu= \mu_1 \otimes \ldots \otimes \mu_n$ be the standard Gaussian measure on $\R^n$. Let $f: \R^n\to \sym$ be such that all of its matrix coefficients are smooth and in $\LL$.  Then for any  $t \in \R_+$,
\[
\mu\bigg( \lambda_{max}\big(f-\E_\mu f\big)\geq t\bigg)\leq 
d\exp\bigg(-\frac{t^2}{2 v_f+ t\sqrt{2 v_f}}\bigg),
\]
where $v_f=\displaystyle \sup_{(x_1,\ldots,x_n)\in \R^n} \big\| \sum_{i=1}^n (\partial_i f)^2(x_1, \cdots, x_n)\big\|$ with $\partial_i f (x_1, \ldots, x_n) =\frac{\partial f}{\partial x_i} (x_1, \ldots, x_n)$.
\end{theorem}

The interesting feature in the above theorem is that it captures concentration in terms of the variations of the matrix function, in the usual sense of variations. We couldn't locate a comparable result in the literature, as previous matrix concentration inequalities designed specifically for the Gaussian measure dealt with matrix Gaussian series; i.e.  $f(x_1,\ldots,x_n)=\sum_{i=1}^n x_i A_i$ for some deterministic $A_1,\ldots,A_n\in \sym$.  
 We should note that for this particular example, the literature contains sub-Gaussian bounds on matrix concentration (see \cite[Chapter~4]{Tropp-survey}).

In Section~\ref{sec: product measures}, we further illustrate this procedure by investigating 
general product measures. To this aim, we prove a corresponding matrix Poincar\'e inequality  (Theorem~\ref{th:Poincare-PM}) and derive an exponential matrix concentration inequality (Theorem~\ref{th: product measures}). In this setting, the matrix Poincar\'{e} inequality is equivalent to the matrix Efron-Stein inequality \cite[Theorem~5.1]{Ch-Hs}. We provide an alternative proof of this by building an appropriate Markov process.  

We should note that, as in the scalar case, the approach based on Poincar\'e inequalities cannot lead to sub-Gaussian bounds on concentration. A possible  approach to deriving sub-Gaussian bounds would be the elaboration of matrix log-Sobolev inequalities and of a general procedure turning these into corresponding matrix concentration inequalities. As of this writing, such procedure remains a challenging task and it is not clear how it could be implemented.

\subsection*{\bf Matrix Poincar\'e and concentration for SCP measures}
Concentration inequalities become increasingly more challenging without the independence structure. The matrix setting adds another layer of difficulty to the problem.
In view of this, it is not surprising that there are few matrix concentration inequalities in the dependent case with \cite{BMY-betamixing,MatrixCon-exch-pairs} only dealing with sums of random matrices, while in the works \cite{PMT, PMT-EffronStein} a matrix bounded difference inequality was established under a form of weak dependence.

The interesting feature in Theorem~\ref{th: poincare to concentration} is its ability to deal with any probability measure $\mu$, the main remaining task lies in the construction of a suitable Markov process having $\mu$ as its stationary measure. With this perspective in mind, 
we are able to establish a matrix concentration inequality for functions of negatively dependent random variables. More precisely, we prove in Theorem \ref{th: strong rayleigh} a matrix Poincare inequality for any homogeneous probability measure on the $n$-dimensional unit cube satisfying a form of negative dependence known as the \emph{stochastic covering property} (SCP).  Combined with Theorem \ref{th: poincare to concentration}, this implies a corresponding matrix exponential concentration inequality. In the scalar case, sub-Gaussian concentration bounds were obtained by Pemantle--Peres \cite{Pemantle-Peres} and Hermon--Salez \cite{H-Salez} who also established a modified log-Sobolev inequality. The proof of Theorem~\ref{th: strong rayleigh} relies on the approach of \cite{H-Salez} and extends it to the matrix setting.

The stochastic covering property was put forward in \cite{Pemantle-Peres} as a form of negative dependence. Indeed, it was shown in \cite{Pemantle-Peres} that the  \emph{strong Rayleigh property} implies SCP. The class of strong Rayleigh measures was introduced by Borcea--Br\"and\'en--Liggett in \cite{BBL} with the aim of building a theory of negative dependence. One of the main features of this class is its stability under many natural operation such as conditioning and projecting. Due to this, the strong Rayleigh property, which implies negative association, is more commonly used. Moreover,   the class of strong Rayleigh measures (and thus the ones satisfying SCP) contains  several interesting examples, such as: determinantal measures and point processes, independent Bernoullis conditioned on their sum, measures obtained by running the exclusion dynamics from a deterministic state.

We begin by recalling the definition of the stochastic covering property.  Let $n\in \N$. We equip the $n$-dimensional discrete unit cube $\{0,1\}^n$ with the partial order $\preceq$ defined by 
  $$x \preceq y \Longleftrightarrow x=y\,\,\,\,\textrm{or}\,\,\,\,\exists !\,\,i\leq n; y_i=x_i+1.$$
  We extend this order to the set of probability measures on $\{0,1\}^n$ in the following way.  If $\mu_1$ and $\mu_2$ are two probability measures on $\{0,1\}^n$,  we write $\mu_1\preceq \mu_2$ if there exists a coupling $\kappa$ on $\{0,1\}^n\times \{0,1\}^n$  such that  $\textrm{support}(\kappa)\subseteq \{(x,y)\in \{0,1\}^n\times \{0,1\}^n:\, x\preceq y\}$. 
  
Fix now $k\in\{1, \cdots ,n\}$. Let $\mu$ be a probability measure on $\{0,1\}^n$ and $\xi$ be a random variable on $\Omega$ with distribution $\mu$. We say that $\mu$ is  $k$-\emph{homogeneous} if   $\Prob(\sum_{i=1}^n{\xi_i}=k)=1$, that is, $\mu$ is a probability measure on $\Omega:=\{x\in \{0,1\}^n:\, \sum_{i=1}^n x_i=k\}$. We say that $\mu$ has the \emph{stochastic covering property}  if  for any subset $S$ of $\{1,\cdots, n\}$, and any  $x_S,y_S\in \{0,1\}^S$,  
$$x_S \preceq y_S \Longrightarrow \mu\left(\cdot \mid \xi_S=y_S\right) \preceq \mu\left( \cdot \mid \xi_S=x_S\right),$$
 where $\xi_S$ denotes the restriction of $\xi$ to the coordinates indexed by $S$. We should note that $\mu\left(\cdot \mid \xi_S=y_S\right)$ is a probability measure on $\{0,1\}^{S^c}$ supported on vectors $z_{S^c}$ satisfying $\Vert z_{S^c}\Vert_1=k-\Vert y_S\Vert_1$, where $\Vert\cdot\Vert_1$ stands for the $\ell_1$-norm (here equal to the number of coordinates equal to one). We are now able to state the matrix Poincar\'{e} inequality for SCP measures. 
\begin{theorem}\label{th: strong rayleigh}
Let $\mu$ be a $k$-homogeneous probability measure on $\Omega:=\{0,1\}^n$  with the SCP property and denote by $\widetilde{\Omega}=\{(x,y)\in\Omega^2:\, \text{$x$ and $y$ differ on exactly $2$ coordinates}\}$. Then there exists a Markov generator $Q$ supported on $\widetilde{\Omega}$ and satisfying $\max\{-Q(x,x):\, x\in \Omega\}\leq 1$, such that for any $f:\, \Omega\to \sym$, we have 
$$
\Var_\mu(f)\preceq 2k \Dir(f),
$$
where $\Dir$ is the Dirichlet form associated with $Q$. 
\end{theorem}

We refer to Section~\ref{sec: strong rayleigh} where an explicit expression of the Markov generator is given. The result above states that any probability measure with the SCP property satisfies a matrix Poincar\'e inequality with normalized Markov generator and constant $2k$. The analogous result in the scalar case  was recently established in 
\cite{H-Salez}. While in \cite{H-Salez} a scalar Poincar\'e inequality (and a modified log-Sobolev inequality) is derived by means of an induction method introduced in \cite{yau}, we extract from this inductive procedure the  explicit Markov generator and prove the matrix Poincar\'e inequality directly using operator convexity arguments.  
Combined with Theorem~\ref{th: poincare to concentration}, this implies the following matrix concentration inequality. 
\begin{theorem}\label{th: SR-conc}
Let $\mu$ be a $k$-homogeneous probability measure on $\Omega=\{0,1\}^n$ with the SCP property. Let $f: \Omega\to \sym$ be a $1$-Lipschitz matrix function in the sense that for any $x,y\in \Omega$,  
$$
\Vert f(x)-f(y)\Vert \leq \Vert x-y\Vert_1, 
$$
where $\Vert \cdot \Vert$ stands for the operator norm. Then, for any $t \in \R_+$,
\[
\mu\bigg( \lambda_{max}\big(f-\E_\mu f\big)\geq t\bigg)\leq 
d\exp\bigg(-\frac{t^2}{8k+ 2t\sqrt{2k}}\bigg).
\]
\end{theorem}
\medskip 

A special case when $f(x_1,\ldots,x_n)=\sum_{i=1}^n x_i A_i$ for some $A_1,\ldots,A_n\in \sym$, was recently investigated in \cite{KS} where, up to a logarithmic term, a sub-Gaussian concentration bound is derived. The method developed in \cite{KS} extends to the matrix setting the martingale approach elaborated in \cite{Pemantle-Peres}. Once again, while our matrix Poincar\'e inequality cannot yield sub-Gaussian concentration, it provides a matrix concentration valid for any matrix Lipschitz function while the results in \cite{KS} are only concerned with the specific example provided by $f(x_1,\ldots,x_n)=\sum_{i=1}^n x_i A_i$ for some $A_1,\ldots,A_n\in \sym$.

The paper is organized as follows. In Section~\ref{sec: matrix generator}, we recall the notion of matrix Markov generator, its associated matrix Dirichlet form and carr\'e du champ operator and their properties. In section~\ref{sec: matrix trace ineq}, we establish the relation between the matrix Dirichlet form and the Laplace transform. Section~\ref{sec: poincare to concentration} is devoted to the proof of Theorem~\ref{th: poincare to concentration}. In Section~\ref{sec: product measures}, we investigate the matrix Poincar\'e inequality for product measures and in particular for the standard Gaussian measure. Finally, Section~\ref{sec: strong rayleigh} is devoted to the study of the matrix Poincar\'e inequality for SCP measures.

\subsection*{Aknowledgement:} 

The authors are grateful to the anonymous referee whose remarks and generous suggestions  greatly improved the presentation. In particular,  the current proofs of Theorem~3.3 and Lemma~3.4 were communicated  by the referee   and replace previously lengthy arguments. 
The authors would also like to thank De Huang and Lester Mackey for helpful comments. 
This work was initiated during visits of the authors to Universit\'e Paris Diderot and the American University of Beirut. We would like to thank these institutions for the excellent working conditions. P.Y. was supported by grant ANR-16-CE40-0024-01.

\section{Matrix markov semi-group and generator}\label{sec: matrix generator}

In  \cite[Section 3]{Ch-Hs-17}, the authors developed a framework of Markov semigroups on matrix-valued functions and defined a non-commutative
version of the carr\'e du champ operator and Dirichlet form. 
In this preliminary section, we recall these notions in our context for completeness and state the necessary properties which we will use. 
We refer to \cite{Ch-Hs-17} for more on this topic. 

Let $\Omega$ be a Polish space and $(X_t)_{t\geq 0}$ be a Markov process with stationary measure $\mu$. Let  $\LL$ be the Hilbert space of square integrable functions with respect to $\mu$. The Markov semi-group $(P_t)_{t\geq 0}$ associated to $(X_t)_{t\geq 0}$ defines an operator on $\LL$ through the formula  
$$
P_t f(x)=\E[ f(X_t)\mid X_0=x],
$$
for any $x\in \Omega$.  Recall that 
the Markov process is said to be reversible if for every $f,g\in \LL$, $$\langle f, P_t g \rangle_{\mu}=\langle P_t f, g \rangle_{\mu},$$
where $\langle \cdot, \cdot\rangle_{\mu}$ is the usual inner product of $\LL$. 
It is said to be ergodic if $P_t f \underset{t\to +\infty}{\longrightarrow} \E_{\mu}(f)$ in $\LL$ for every $f\in \LL$.

To the Markov semi-group $(P_t)_{t\geq 0}$ is associated its     infinitesimal generator $\gen$   defined by$$
\gen f=\lim_{t\to 0} \frac{P_t f-f}{t},
$$
for every function $f\in \DD$, where $\DD$ is the $\LL$-domain of $\gen$.  
In this paper, we are mostly interested in (Hermitian) matrix valued functions. The action of the Markov semi-group and that of the infinitesimal generator can be naturally extended to matrix valued functions by considering the action of the semi-group on each entry of the matrix valued function.  
Therefore, given a  function $f: \Omega \to \sym$ whose matrix coefficients belong to $\LL$ (or to $\DD$), we set similarly 
$$
P_t f(x)=\E[ f(X_t)\mid X_0=x], 
$$ 
for any $x\in \Omega$ and $\gen f=\lim_{t\to 0} \frac{P_tf-f}{t}$. 
We will refer to $P_t$ as \textit{matrix Markov semi-group} and $\gen$ as \textit{matrix Markov generator} to emphasize that we will be interested in their  action on matrix valued functions. By abuse of notation, the space of  functions $f: \Omega  \to \sym$ whose matrix coefficients belong to $\DD$  will be still denoted by $\DD$.    Below, we collect some of the basic properties paralleling their scalar counterpart and refer to \cite{Ch-Hs-17} for their proofs. 

\begin{proposition}\label{lem: matrix generator prop}
Let $\Omega$ be a Polish space and $(X_t)_{t\geq 0}$ be a Markov process with stationary measure $\mu$. The matrix Markov semi-group $(P_t)_{t\geq 0}$ and matrix Markov generator $\gen$ satisfy the following elementary properties for every matrix valued functions $f,g\in \DD$:
\begin{enumerate}
    \item $P_t$ and $\gen$ commute.
        \item If $(X_t)_{t\in \mathbb{R}_+}$ is reversible then $\E_\mu[f\gen(g)]=\E_\mu[\gen(f) g]$.  
    \item  $\E_\mu[\gen f]=0$.
    \item   $P_tf$ takes values in $\PSD$.
   \item \label{en:square-MSG} $(P_tf)^2\preceq P_tf^2$. 
    \item \label{en:Jenson-MSG}  If $\phi:\R\to \R$ is a convex function, then 
    $$
    \Tr\big(\phi(P_t f)\big)\leq \Tr\big(P_t\phi(f)\big).
    $$
\end{enumerate}
\end{proposition}

We introduce the \textit{matrix Dirichlet form} given by 
$$
\Dir(f)=- \E_\mu[f\gen(f)],
$$
for any $f\in \DD$. Given an algebra $\mathcal{A}\subseteq \DD$, define    the \textit{matrix carr\'e du champ} operator   by 
 $$\Gamma(f):= \frac12 \big( \gen(f^2)- f\gen(f) - \gen(f) f\big),
$$
for $f\in \mathcal{A}$, 
Let us note that already in the above definitions, we see the subtlety of the noncommutative nature of the objects manipulated. For instance, while in the scalar counterpart $f$ and $\gen(f)$ commute, this is no longer the case here and one needs to take into account this when dealing with the above notions.

It is not clear at first glance if the usual properties of the Dirichlet form and the carr\'e du champ operator extend to their matrix counterparts. This is the case when the underlying Markov process is reversible as we verify in the next proposition. 

\begin{proposition}\label{prop: dirichlet properties}
Let $\Omega$ be a Polish space and $(X_t)_{t\geq 0}$ be a reversible Markov process with stationary measure $\mu$. Then the matrix Dirichlet form $\Dir$ and the carr\'e du champ operator $\Gamma$ satisfy the following properties for every matrix valued function $f\in \mathcal{A}$:
\begin{enumerate}
    \item For any $a\in \R$, we have $\Gamma(af)=a^2\, \Gamma(f)$. 
    \item $\Dir(f)=\E_\mu[\Gamma(f)]$.
    \item We have 
    $$
    \Gamma(f)(x)=\lim_{t\to 0} \frac{\E\big[\big(f(X_t)-f(X_0)\big)^2\mid X_0=x\big]}{2t}
    $$
    and 
      $$
    \Dir(f)=\lim_{t\to 0} \frac{\E\big[\big(f(X_t)-f(X_0)\big)^2\big]}{2t}, \quad  X_0 \sim  \mu.   $$
    In particular, $\Gamma(f):\Omega\to \PSD$ and $\Dir(f)\in\PSD$. 
    \item If $\Omega$ is finite, then for any $x\in \Omega$ we have 
    $$
    \Gamma(f)(x)= \frac12 \sum_{y\in \Omega} Q(x,y) \big(f(y)-f(x)\big)^2,
    $$
    and 
    $$
    \Dir(f)=\frac12 \sum_{x,y\in \Omega} \mu(x)Q(x,y) \big(f(y)-f(x)\big)^2,
    $$
    where $Q$ is the $\vert \Omega\vert\times \vert \Omega\vert$ matrix representing the Markov generator. 
\end{enumerate}
\end{proposition}

\begin{proof}
The first property follows easily from the linearity of $\gen$. 
To establish the second property, note that 
$$
\E_\mu [\Gamma(f)]=  \frac12 \big( \E_\mu[\gen(f^2)]-\E_\mu[f\gen(f)]-\E_\mu[\gen(f) f]\big).
$$
Using that $\E_\mu[\gen(f^2)]=0$ and that $\E_\mu[f\gen(f)]=\E_\mu[\gen(f) f]$ by reversibility, we get the claim. To establish the third, we use the definition of $\gen$ to explicitly write   
\begin{align*}
\Gamma(f)(x)&=\lim_{t\to 0} \frac{\E[f^2(X_t)\mid X_0=x]-f^2(x)}{2t}- \lim_{t\to 0}\frac{f(x) \big(\E[f(X_t)\mid X_0=x]- f(x)\big)}{2t}\\
&\\
&\qquad -\lim_{t\to 0}\frac{\big(\E[f(X_t)\mid X_0=x]- f(x)\big) f(x)}{2t} \\
&\\
&= \lim_{t\to 0} \frac{\E[f^2(X_t)\mid X_0=x]+f^2(x)- f(x)\E[f(X_t)\mid X_0=x]-\E[f(X_t)\mid X_0=x]f(x)}{2t}\\
&=\lim_{t\to 0} \frac{\E\big[\big(f(X_t)-f(X_0)\big)^2\mid X_0=x\big]}{2t},
\end{align*}
which establishes the claim for $\Gamma(f)$. 
Since $\Dir(f)=\E_\mu[\Gamma(f)]$, the expression of $\Dir(f)$ readily follows. 
From these representations, it is clear that $\Gamma(f)$ and $\Dir$ take values in $\PSD$.

Finally, when $\Omega$ is finite, using the above and defining $h_x: \Omega \to \sym$ by $h_x(y)=\big(f(y)-f(x)\big)^2$, we get 
$$
\Gamma(f)(x)=\lim_{t\to 0} \frac{(P_th_x)(x)}{2t}= \frac12 (Qh_x)(x)= \frac12 \sum_{y\in \Omega} Q(x,y) h_x(y),
$$
which proves the expression of $\Gamma(f)$ in last property of the proposition. It remains to use that $\Dir(f)=\E_\mu[\Gamma(f)]$ to derive the expression of $\Dir(f)$ and finish the proof. 
\end{proof}

We collect furthermore  some useful identities connecting the variance and the Dirichlet form. Unlike the scalar case, the following identity requires reversibility of the Markov  process. 

\begin{lemma}\label{lem:var-Dir}
Let $\Omega$ be a Polish space and $(X_t)_{t\geq 0}$ be a \emph{reversible} Markov process with stationary measure $\mu$. Then for any $f\in \DD$, we have
\[
\frac{d}{dt} \Var_\mu (P_t f)=-2 \Dir (P_t f).
\]
Moreover, if the Markov process is ergodic then
\[
\Var_\mu (f)= 2 \int_0^\infty \Dir (P_t f) \, dt. 
\]
\end{lemma}

\begin{proof}
We start by proving the first identity. As  $\E_\mu [P_t f ]= \E_\mu [f]$ then
\begin{align*}
\frac{d}{dt} \Var_\mu [P_t f] = \frac{d}{dt} \E_\mu[ (P_t f)^2 ] = \E_\mu \Big[ \Big( \frac{d}{dt} P_t f \Big) P_t f \Big] + \E_\mu \Big[ P_t f   \Big( \frac{d}{dt} P_t f \Big)\Big] . 
\end{align*}
Noting that $\frac{d}{dt} P_t f = \gen (P_t f)$ and using reversibility, we get
\begin{align*}
\frac{d}{dt} \Var_\mu [P_t f] = \E_\mu \Big[  \gen ( P_t f) P_t f \Big] + \E_\mu \Big[ P_t f   \gen ( P_t f)\Big] = 2   \E_\mu \Big[ P_t f   \gen ( P_t f)\Big] = - 2 \Dir (P_t f). 
\end{align*}
In order to prove the second identity, we shall first prove that $\Var_{\mu} [P_t f]$ converges to zero as $t$ goes to infinity. In fact, the  ergodicity of the Markov process implies that, for any $f:\Omega \to \sym$ and any $i,j\in\{1,\ldots, d\}$, $(P_t f)_{ij}$ converges to $(\E_\mu[f])_{ij}$ in $L^2(\mu)$.  Then for any $i,j =1, \ldots d$, using Cauchy-Schwarz inequality 
\begin{align*}
\big\vert (\Var_\mu [P_t f])_{ij}\big\vert &= \Big\vert\sum_{k=1}^d \E_{\mu}[ (P_t f -\E_\mu[f])_{ik}(P_t f -\E_\mu[f])_{kj}]\Big\vert \\
&\leq \sum_{k=1}^d \big(\E_\mu[ (P_t f -\E_\mu[f])_{ik}^2]\big)^{\frac12}\big(\E_\mu[(P_t f -\E_\mu[f])_{kj}^2]\big)^{\frac12} \xrightarrow[t\rightarrow \infty]{}0
\end{align*}
and hence $\Var_\mu[P_t f] \xrightarrow[t\rightarrow \infty]{} 0 $. 
Therefore, we get by the fundamental theorem of calculus 
\[
\Var_\mu [f] = \Var_\mu [P_0 f] - \lim_{t\rightarrow \infty} \Var_\mu[P_t f]=- \int_0^\infty \frac{d}{dt} \Var_\mu [P_t f] dt= 2 \int_0^\infty \Dir (P_t f, P_t f) \, dt. 
 \]
\end{proof}

\section{Some matrix trace inequalities}\label{sec: matrix trace ineq}

The goal of this section is to establish the following trace inequality relating the matrix Dirichlet form to the Laplace transform. 

\begin{theorem}\label{th: dir-laplace}
Let $\Omega$ be a Polish space and $(X_t)_{t\geq 0}$ be a reversible Markov process with stationary measure $\mu$. Then, for any  $g: \Omega\to \sym$ belonging to the domain of $\Gamma$ and any $p\geq 1$,
$$
\Tr\Big[\big(\Dir(e^g)\big)^p\Big]\leq   \Big\Vert \Vert \Gamma(g)\Vert^p\Big\Vert_{L_\infty}\,  \Tr\big(\E_\mu[e^{2pg}]\big),
$$
where $\Dir$ and $\Gamma$ are respectively  the matrix Dirichlet form and matrix carr\'e du champ operator associated with $(X_t)_{t\geq 0}$. 
\end{theorem}

The above theorem is the cornerstone of the argument relating the matrix poincar\'e inequality to matrix concentration.  In order to prove this statement, we establish some matrix trace inequalities which we believe might be of independent interest. Before stating and proving these inequalities, let us mention a key ingredient which we will rely on. 
The following statement is a particular case of the results in \cite{Hiai-Kosaki}, and provides an integral representation of the matrix logarithmic mean in terms of the matrix arithmetic mean. 

\begin{theorem}\cite[Corollary~2.4]{Hiai-Kosaki}\label{th: hiai-kosaki}
There exists a probability measure $\nu$ on $\R$ such that the following holds. 
Let $H, K$ be two $d\times d$ positive definite matrices and let $X\in \sym$. Then 
$$
\int_0^1 H^\tau XK^{1-\tau}\, d\tau= \int_{\R} H^{is} \frac{HX+XK}{2} K^{-is} \, d\nu (s).
$$
\end{theorem}

The probability measure $\nu$ in the above statement can be made explicit (see equation~2.7 in \cite{Hiai-Kosaki}), however this won't be needed for our purposes. 
Let us mention that the above is only one of several integral representations established by Hiai and Kosaki in \cite{Hiai-Kosaki}.

Let us recall that given a function $f:\, \R\to \R$, it can be extended to a function on Hermitian matrices by applying it to each eigenvalue of the matrix. 
More precisely, if $A= \sum_{i=1}^d \lambda_i v_iv_i^*$ is the spectral decomposition of $A\in \sym$, then one defines $f(A)= \sum_{i=1}^d f(\lambda_i) v_iv_i^*$. 
The main trace inequality used to prove Theorem~\ref{th: dir-laplace} is the following.

\begin{theorem}\label{th: mean value}
Let $A$ be a $d\times d$ Hermitian (deterministic) matrix and let $B$ be a $d\times d$ Hermitian random matrix. Then for 
every increasing, convex function $f: \R_+ \to \R_+$, 
$$
\Tr\Big[f\Big(\E[(e^A-e^B)^2]\Big)\Big]\leq \frac12 \Tr\Big[f\big(\E[(A-B)e^{2B}(A-B)]\big)\Big]+\frac12\Tr\Big[f\big(e^{A}\E[(A-B)^2]e^{A}\big)\Big]. 
$$
In particular, the above holds for $f:\, t\to t^p$ when $p\geq 1$. 
\end{theorem}

\begin{proof}
First, note that 
$\frac{d}{d\tau} e^{\tau A}e^{(1-\tau)B}= e^{\tau A}(A-B)e^{(1-\tau)B}$. Therefore, we can write 
$$
e^A-e^B= \int_{0}^1\frac{d}{d\tau}  e^{\tau A}e^{(1-\tau)B}\, d\tau= \int_{0}^1e^{\tau A}(A-B)e^{(1-\tau)B} \, d\tau.
$$
Denoting $C:=A-B$ and using Theorem~\ref{th: hiai-kosaki}, we have 
$$
e^A-e^B= \int_{\R} e^{isA} \frac{e^{A}C+Ce^{B}}{2} e^{-isB} \, d\nu (s):=  \int_{\R} \Gamma(s) \, d\nu (s).
$$
Using \cite[Corollary~2.8]{Choi}, we have that
$$
\int_{\R} \Gamma(s) \, d\nu (s) \cdot \int_{\R} \Gamma(s)^* \, d\nu (s) \preceq \int_{\R} \Gamma(s) \Gamma(s)^*\, d\nu (s).
$$
Putting the above together, we get 
\begin{align*}
(e^A-e^B)^2&\preceq  \int_{\R} e^{isA} \Big(\frac{e^{A}C+Ce^{B}}{2}\Big)\cdot  \Big(\frac{e^{A}C+Ce^{B}}{2}\Big)^* e^{-isA} \, d\nu (s)\\
&\preceq \int_{\R} e^{isA} \frac{e^{A}C^2e^{A}+Ce^{2B}C}{2} e^{-isA} \, d\nu (s),
\end{align*}
where the last inequality follows after using that $(e^{A}C+Ce^{B})(e^{A}C+Ce^{B})^*\preceq 2(e^{A}C)(e^{A}C)^*+2(Ce^{B})(Ce^{B})^*$. 

Since $f$ is continuous, convex monotone increasing, then so is $\Tr\circ f$ (see for instance \cite[Theorem~2.10]{Carlen}). Therefore, recalling that $A$ is deterministic, 
we have 
\begin{align*}
\Tr\Big[f\Big(\E[(e^A-e^B)^2]\Big)\Big]&\leq  \Tr\Big[ f\Big(\int_{\R} e^{isA} \E\Big[\frac{e^{A}C^2e^{A}+Ce^{2B}C}{2}\Big] e^{-isA} \, d\nu (s)\Big)\Big]\\
&\leq  \int_{\R} \Tr\Big[ f\Big(e^{isA} \E\Big[\frac{e^{A}C^2e^{A}+Ce^{2B}C}{2}\Big] e^{-isA} \Big)\Big]\, d\nu (s),
\end{align*}
where the first inequality uses the monotonicity of $\Tr\circ f$, and the second its convexity. 
Now note that $e^{isA}$ is unitary, therefore using that $\Tr\circ f$ is unitary invariant, we get 
\begin{align*}
\Tr\Big[f\Big(\E[(e^A-e^B)^2]\Big)\Big]&\leq  \Tr\Big[ f\Big(\E\Big[\frac{e^{A}C^2e^{A}+Ce^{2B}C}{2}\Big] \Big)\Big]\\
&\leq  \frac12 \Tr\Big[f\big(\E[(A-B)e^{2B}(A-B)]\big)\Big]+\frac12\Tr\Big[f\big(e^{A}\E[(A-B)^2]e^{A}\big)\Big],
\end{align*}
where the last inequality follows by convexity of $\Tr\circ f$. 
\end{proof}

The above theorem seems new even when both $A$ and $B$ are deterministic. For instance, we get that for any $p\in \N$
$$
\Tr\Big[(e^A-e^B)^{2p}\Big]\leq \frac12\Tr\Big[\big((A-B)^2e^{2B}\big)^{p}\Big]+\frac12\Tr\Big[\big(e^{2A}(A-B)^2\big)^{p}\Big]. 
$$
The case $p=1$ is related to some of the results in \cite{PMT-EffronStein} even though it is incomparable to them and cannot be derived from them.  

To deduce Theorem~\ref{th: dir-laplace} from Theorem~\ref{th: mean value}, we will need the following lemma which relies on an operator convexity inequality from \cite{Hansen-Pedersen}. 

\begin{lemma}\label{lem: operator convexity}
Let $K\in \sym$ and $Z\in \PSD$ be  random matrices, not necessarily independent, and assume that $\E[ K^2]\preceq Id$. Then for any $p\geq 1$, we have 
$$
\E[KZK]\preceq \big(\E[KZ^pK]\big)^{\frac{1}{p}}.
$$ 
In particular, 
$$
\Tr\Big[\big(\E[KZK]\big)^p\Big] \leq \E \Tr[KZ^pK].
$$
\end{lemma}
\begin{proof}
By a truncation argument, we may suppose without loss of generality that $Z$ is uniformly bounded. 
Let $\mathcal{A}$ be the $C^*$-algebra of uniformly bounded random Hermitian $d\times d$ matrices and let $K\in \sym$ be such that $\E[K^2]\preceq Id$.  Consider  $\Phi: \, \mathcal{A}\to \sym$ defined by 
$$
\Phi(Z)= \E[KZK].
$$
Clearly, $\Phi$ is linear and positive i.e. $\Phi(Z)\succeq 0$ if $Z\succeq 0$, and $\Phi(0)=0$. Moreover, if $Z$ is uniformly bounded by $1$, then using that $\E[K^2]\preceq Id$ we see that the operator norm of $\Phi(Z)$ is bounded by $1$. Thus, $\Phi$ is a contraction and we can apply \cite[Corollary~2.2]{Hansen-Pedersen} to deduce that for any $p\geq 1$ and any $Z\in \mathcal{A}$ 
$$
\Phi(Z^{\frac{1}{p}})\preceq \Phi(Z)^{\frac{1}{p}},
$$
where we have used that the function $t\to t^{\frac{1}{p}}$ is operator concave for $p\geq 1$. The first claim then follows after a change of variables, while the second uses that 
$\Tr\circ (\cdot)^p$ is operator monotone. 
\end{proof}

With the help of Lemma~\ref{lem: operator convexity}, we are now ready to show how Theorem~\ref{th: mean value} implies Theorem~\ref{th: dir-laplace}. 

\begin{proof}[Proof of Theorem~\ref{th: dir-laplace}]
Using Proposition~\ref{prop: dirichlet properties}, we start writing 
$$
\Tr\Big[\big(\Dir(e^g)\big)^p\Big]=\lim_{t\to 0}\frac{1}{(2t)^p} \Tr\Big[\big(\E\big(e^{ g(X_t)}-e^{ g(X_0)}\big)^2\big)^p\Big].
$$
Let $(A,B):=\big(g(X_0),g(X_t)\big)$ and note that it follows from the reversibility of $(X_t)_{t\geq 0}$ that $(A,B)$ is an exchangeable pair. Since $\Tr [(\cdot)^p]$ is operator convex, then Jensen's inequality implies that 
\begin{equation}\label{eq: proof-dir-laplace}
\Tr\Big[\big(\Dir(e^g)\big)^p\Big]\leq \lim_{t\to 0}\frac{1}{(2t)^p} \E\Tr\Big[\big(\E\big[\big(e^{A}-e^{B}\big)^2\mid A\big]\big)^p\Big].
\end{equation}
Denote $\alpha=\Big\Vert \E[(A-B)^2\mid A]\Big\Vert_{L^\infty(A)}$. Applying Theorem~\ref{th: mean value} conditionally on $A$, we get 
$$
\E \Tr\Big[\big(\E[(e^A-e^B)^2\mid A]\big)^{p}\Big]\leq
\frac12 \E\Tr\Big[\big(\E[(A-B)e^{2B}(A-B)\mid A]\big)^{p}\Big]+\frac12 \Tr\Big[\big(e^{A}\E[(A-B)^2\mid A]e^{A}\big)^{p}\Big].
$$
Now note that $e^{A}\E[(A-B)^2\mid A]e^{A}\preceq \Vert \E[(A-B)^2\mid A]\Vert\, e^{2A}$ for any realization of $A$. Therefore, using the monotonicity of $\Tr[(\cdot)^p]$, we get that 
$$
\Tr\Big[\big(e^{A}\E[(A-B)^2\mid A]e^{A}\big)^{p}\Big] 
\leq \alpha^p\, \Tr(e^{2pA}),
$$
for all realizations of $A$. 
On the other hand, using Lemma~\ref{lem: operator convexity} conditionally on $A$ with $K= \frac{A-B}{\sqrt{\alpha}}$ and $Z=e^{2B}$, we have 
$$
\Tr\Big[\big(\E[(A-B)e^{2B}(A-B)\mid A]\big)^{p}\Big]
\leq \Vert\E[(A-B)^2\mid A]\Vert^{p-1} \Tr\Big[\E[(A-B)e^{2pB}(A-B)\mid A]\Big],
$$
for all realizations of $A$. Combining these estimates, we deduce that 
\begin{align*}
\E\Tr\Big[\big(\E[(e^A-e^B)^2\mid A]\big)^{p}\Big]&\leq
\frac12 \alpha^{p-1} \E\Tr\Big[\E[(A-B)e^{2pB}(A-B)\mid A]\Big]+ \frac12 \alpha^p\, \Tr(e^{2pA})\\
&= \frac12\alpha^{p-1}\E \Tr\Big[(A-B)^2e^{2pB}\Big]+\frac12\alpha^{p}\, \E\Tr(e^{2pA}).
\end{align*}
Using the exchangeability of $A$ and $B$, we have 
$$
\E \Tr\Big[(A-B)^2e^{2pB}\Big]\leq \alpha\,  \E\Tr(e^{2pA}).
$$
Combining the two previous estimates, we deduce that
$$
\E\Tr\Big[\big(\E[(e^A-e^B)^2\mid A]\big)^{p}\Big]\leq \Big\Vert \E[(A-B)^2\mid A]\Big\Vert_{L^\infty}^p\, \E\Tr(e^{2pA}).
$$ 
Replacing $A$ and $B$ by their values, and plugging back the above inequality in \eqref{eq: proof-dir-laplace}, we get 
$$
\Tr\Big[\big(\Dir(e^g)\big)^p\Big]\leq  \Big\Vert \lim_{t\to 0}\frac{1}{2t} \E[\big(g(X_t)-g(X_0)\big)^2\mid X_0]\Big\Vert_{L^\infty}^p\, \E_\mu\Tr(e^{2pg}).
$$
Using Proposition~\ref{prop: dirichlet properties}, we have 
$$
 \lim_{t\to 0}\frac{1}{2t} \E[\big(g(X_t)-g(X_0)\big)^2\mid X_0]= \Gamma(g),
$$
which after replacement in the previous inequality, finishes the proof. 
\end{proof}

\section{From matrix Poincar\'e inequality to matrix concentration}\label{sec: poincare to concentration}

The goal of this section is to prove Theorem~\ref{th: poincare to concentration}. Like usual, providing a bound on the Laplace transform will be sufficient to derive the corresponding concentration inequality through the use of a Chernoff bound. More precisely, we prove the following. 

\begin{theorem}\label{th: laplace transform}
Let $\mu$ be a probability measure on some Polish space $\Omega$. Suppose that $\mu$ satisfies a matrix Poincar\'e inequality with constant $\alpha$ and matrix Markov generator $\gen$ reversible with respect to $\mu$. Given $f: \Omega\to \sym$ belonging to the domain of $\Gamma$, denote 
$$
v_f=\sup \big\Vert \Gamma(f)\big\Vert.$$
Suppose that $e^{\lambda f}$ has matrix coefficients in $\rm{L}^1(\mu)$ for any $\lambda>0$. 
Then, for any $\delta< \sqrt{\frac{2}{\alpha v_f}}$, we have 
$$
\E_{\mu}\Tr(e^{\delta(f-\E_\mu f)})\leq \frac{2d}{2-\alpha v_f\delta^2}.
$$
\end{theorem}

With this bound in hand, we can easily derive the concentration inequality promised in Theorem~\ref{th: poincare to concentration}. 
\begin{proof}[Proof of Theorem~\ref{th: poincare to concentration}]
The proof follows the standard method initiated by Ahlswede-Winter \cite{AW}. Using Chernoff bound, we write 
$$
\mu\Big( \lambda_{max}\big(f-\E_\mu  f\big)\geq t\Big)\leq 
\inf_{\delta\in \R_+} e^{-\delta t}\E_{\mu}\Tr(e^{\delta(f-\E_\mu f)}).
$$
Using Theorem~\ref{th: laplace transform}, we get 
$$
\mu\Big( \lambda_{max}\big(f-\E_\mu  f\big)\geq t\Big)\leq 
\inf_{\delta<\sqrt{\frac{2}{\alpha v_f}}} \frac{2de^{-\delta t}}{2-\alpha v_f \delta^2}\leq \inf_{\delta<\sqrt{\frac{2}{\alpha v_f}}} d\cdot e^{-\delta t 
+\frac{\alpha v_f\delta^2}{2-\sqrt{2\alpha v_f} \delta}}.
$$
 Choosing $\delta=\frac{t}{\alpha v_f+ t\sqrt{\alpha v_f/2}}$, we get the result. 
\end{proof}

The rest of this section is devoted to the proof of Theorem~\ref{th: laplace transform}. As explained in the introduction, the key is to establish some recursive relation involving the Laplace transform. 
To this aim, the result of the previous section relating the matrix Dirichlet form to the Laplace transform will play a crucial role in the derivation of such recursive formula. 
We start with the following elementary lemma.

\begin{lemma}\label{lem: weighted convexity}
Let $A$ and $B$ be two Hermitian matrices and $\gamma>1$. Then, for any $p\in \N$, we have 
$$
\Tr\big[(A+B)^p\big]\leq \Big(\frac{\gamma}{\gamma-1}\Big)^{p-1} \Tr(A^p)+ \gamma^{p-1}\Tr(B^p).
$$
\end{lemma}
\begin{proof}
Let $\w A= \frac{\gamma}{\gamma-1} A$ and $\w B=\gamma B$. With these notations, we have 
$$
\Tr\big[(A+B)^p\big]= \Tr\Big[\big(\frac{\gamma-1}{\gamma}\w A+\frac{1}{\gamma}\w B\big)^p\Big].
$$
Using the convexity of $\Tr[(\cdot)^p]$, we get
$$
\Tr\big[(A+B)^p\big]
\leq \frac{\gamma-1}{\gamma} \Tr({\w A}^p)+ \frac{1}{\gamma}\Tr({\w B}^p).
$$
Replacing $\w A$ and $\w B$ by their expressions, we finish the proof. 
\end{proof}

The next lemma will help us implement an induction argument to prove Theorem~\ref{th: laplace transform}. 

\begin{lemma}\label{lem: induction-laplace}
Let $\mu$ be a probability measure on some Polish space $\Omega$. Suppose that $\mu$ satisfies a matrix Poincar\'e inequality with constant $\alpha$ 
and matrix Markov generator $\gen$ reversible with respect to $\mu$. Given $g:\Omega\to \sym$ with $\E_\mu[g]=0$, let $$
v_g= \sup \big\Vert \Gamma(g)\big\Vert.
$$
If $\alpha v_g<1$, then for any $p\in \N$ we have 
$$
\Tr\Big[ \big(\E_\mu [e^{2g}]\big)^{p}\Big] 
\leq \frac{1}{(1- \alpha v_g)^{p-1}}\Tr\Big[ \big(\E_\mu[e^g]\big)^{2p}\Big] 
+\alpha v_g \Tr\Big[ \E_\mu[ e^{2p g}]\Big]. 
$$
\end{lemma}
\begin{proof}
Let $g: \Omega\to \sym$. 
Since $\mu$ satisfies a matrix Poincar\'e inequality, then  
$$
\Var_\mu(e^{g})\preceq \alpha \Dir(e^{g}),
$$
which amounts to 
$$
\E_\mu[e^{2g}]\preceq \big(\E_\mu[e^{g}]\big)^2 + \alpha\Dir(e^{g}).
$$

Using that $\Tr[(\cdot)^p]$ is operator monotone, the above inequality implies that 
$$
\Tr\Big[\big(\E_\mu[e^{2g}]\big)^p\Big]\leq \Tr\Big[\Big(\big(\E_\mu[e^{g}]\big)^2 +  \alpha\Dir(e^{g})\Big)^p\Big].
$$
Now using Lemma~\ref{lem: weighted convexity}, we get 
$$
\Tr\Big[\big(\E_\mu[e^{2g}]\big)^p\Big]\leq 
\Big(\frac{\gamma}{\gamma-1}\Big)^{p-1} \Tr\Big[ \big(\E_\mu[e^g]\big)^{2p}\Big] + \gamma^{p-1}
\alpha^p\Tr\Big[\Big(\Dir(e^{g})\Big)^p\Big], 
$$
where $\gamma>1$ will be chosen in the sequel. 
Using Theorem~\ref{th: dir-laplace}, we deduce that 
$$
\Tr\Big[\big(\E_\mu[e^{2g}]\big)^p\Big]\leq 
\Big(\frac{\gamma}{\gamma-1}\Big)^{p-1} \Tr\Big[ \big(\E_\mu[e^g]\big)^{2p}\Big] + \gamma^{p-1}
\alpha^p v_g^p\, \E_\mu \Tr(e^{2p\, g}). 
$$
It remains to choose $\gamma= (\alpha v_g)^{-1}$ to finish the proof. 
\end{proof}

\begin{proof}[Proof of Theorem~\ref{th: laplace transform}]
Without loss of generality, we assume that $\E_\mu[f]=0$. 
We will implement an induction procedure based on the previous lemma. 
We start by applying Lemma~\ref{lem: induction-laplace} with $g_1=\frac{\delta f}{2}$ and $p=1$ to get 
$$
\Tr\Big[\E_\mu[e^{\delta f}]\Big] \leq 
\frac{1}{1-\alpha v_f(\delta/2)^2} \Tr\Big[\Big(\E_\mu[e^{\frac{\delta f}{2}}]\Big)^2\Big],
$$
where we used that $v_{g_1}= (\delta/2)^2 v_f$. 
Now, apply again Lemma~\ref{lem: induction-laplace} with $g_2= \delta f/2^2$ and $p=2$ to get 
$$
\Tr\Big[\E_\mu[e^{\delta f}]\Big] \leq 
\frac{1}{\big(1-\alpha v_f(\delta/2)^2\big)\big(1-\alpha v_f(\delta/2^2)^2\big)}\Tr\Big[\Big(\E_\mu[e^{\frac{\delta f}{2^2}}]\Big)^{2^2}\Big] + \frac{\alpha v_f (\delta/2^2)^2}{\big(1-\alpha v_f(\delta/2)^2\big)} \Tr\Big[ \E_\mu[ e^{\delta f}]\Big],
$$
which after rearrangement leads to 
\begin{align*}
\Tr\Big[\E_\mu[e^{\delta f}]\Big]
&\leq \frac{1}{\big(1-\alpha v_f(\delta/2)^2-\alpha v_f(\delta/2^2)^2\big)\big(1-\alpha v_f(\delta/2^2)^2\big)}\Tr\Big[\Big(\E_\mu[e^{\frac{\delta f}{2^2}}]\Big)^{2^2}\Big]\\
&\leq 
\frac{1}{\big(1-\alpha v_f(\delta/2)^2-2\alpha v_f(\delta/2^2)^2\big)}\Tr\Big[\Big(\E_\mu[e^{\frac{\delta f}{2^2}}]\Big)^{2^2}\Big]\\
&\leq \frac{1}{\big(1-\alpha v_f\delta^2 (1-2^{-2})/2\big)}\Tr\Big[\Big(\E_\mu[e^{\frac{\delta f}{2^2}}]\Big)^{2^2}\Big].
\end{align*}
We will prove by induction on $s$ that 
$$
\Tr\Big[\E_\mu[e^{\delta f}]\Big] \leq 
\frac{1}{\big(1-\alpha v_f\delta^2(1-2^{-s})/2\big)}\Tr\Big[\Big(\E_\mu[e^{\frac{\delta f}{2^s}}]\Big)^{2^s}\Big].
$$
We verified the above inequality for $s=1$ and $s=2$. Suppose it is true for $s$, and apply Lemma~\ref{lem: induction-laplace} with 
$g= \delta f/2^{s+1}$ and $p=2^s$ to get 
$$
\Tr\Big[\Big(\E_\mu[e^{\frac{\delta f}{2^s}}]\Big)^{2^s}\Big] 
\leq \frac{1}{\big(1- \alpha v_f (\delta/2^{s+1})^2\big)^{2^s-1}}\Tr\Big[ \big(\E_\mu[e^{\delta f/2^{s+1}}]\big)^{2^{s+1}}\Big] 
+ \alpha v_f (\delta/2^{s+1})^2 \Tr\Big[ \E_\mu[ e^{\delta f}]\Big]. 
$$
Combining the above with the induction hypothesis, we get that
$$
\Tr\Big[\E_\mu[e^{\delta f}]\Big] \leq 
\frac{1}{\big(1-2^{-1}\alpha v_f\delta^2(1-2^{-s})- \alpha v_f (\delta/2^{s+1})^2\big)\big(1- \alpha v_f (\delta/2^{s+1})^2\big)^{2^s-1}} \Tr\Big[ \big(\E_\mu[e^{\delta f/2^{s+1}}]\big)^{2^{s+1}}\Big].
$$
Now using that $(1-x)^n\geq 1-nx$ when $x\leq 1$, we deduce from the above that 
$$
\Tr\Big[\E_\mu[e^{\delta f}]\Big] \leq 
\frac{1}{1-2^{-1}\alpha v_f\delta^2(1-2^{-s})- \alpha v_f (\delta/2^{s+1})^2-(2^s-1) \alpha v_f (\delta/2^{s+1})^2} \Tr\Big[ \big(\E_\mu[e^{\delta f/2^{s+1}}]\big)^{2^{s+1}}\Big],
$$
which after a short calculation finishes the induction. 
To finish the proof, take the limit as $s\to \infty$ and notice that 
$$
\Tr\Big[\Big(\E_\mu[e^{\frac{\delta f}{2^s}}]\Big)^{2^s}\Big]\underset{s\to\infty}{\longrightarrow}\Tr\Big[e^{\delta \E_\mu[f]}\Big]= d,
$$
to finish the proof. 
\end{proof}

\section{Matrix Poincar\'e inequality for product measures}\label{sec: product measures}
The aim of this section is to prove Theorems~\ref{th: Gaussian-Poincare} and \ref{thm:Gaussian-prod-Poin}. 
Before doing so, we will investigate general product measures. We will first show that an arbitrary product measure $\mu= \mu_1 \otimes \ldots \otimes \mu_n$  satisfies a suitable matrix Poincar\'e inequality, then  will compute the associated carr\'e du champ operator to deduce the following matrix concentration inequality.

\begin{theorem}\label{th: product measures}
Let $\mu= \mu_1 \otimes \ldots \otimes \mu_n$ be any product measure on some Polish space $\Omega^n$. Let    $f:\Omega^n \to \sym$ be such that 
$$v_f:= \displaystyle\sup_{(x_1,\ldots,x_n)\in \Omega^n} \big\Vert \sum_{i=1}^n \int \big(f(x_1, \ldots, x_n) - f(x_1, \ldots,x_{i-1},z,x_{i+1}, \ldots, x_n)\big)^2 d\mu_i (z)  \big\Vert$$ is finite. Then, for    any $t\geq 0$, we have 
$$
\mu\bigg( \lambda_{max}\big(f-\E_\mu f\big)\geq t\bigg)\leq 
d\exp\Big(-\frac{t^2}{ v_f+ t\sqrt{ v_f}}\Big).
$$
\end{theorem}

The proof of Theorem \ref{th: product measures} simply consists of constructing a  Markov process  with  $\mu$ as stationary measure and having a suitable Markov generator $\gen$ for which we prove a matrix Poincar\'e inequality. 
In this case, the matrix carr\'e du champ operator consists  of the sum of the squared variation in each coordinate of the matrix function.  One of the simplest and most natural smoothness assumptions on a matrix function $f$ is the  so-called bounded difference condition;  i.e.  for any $i=1, \ldots , n$ there exists a deterministic matrix $A_i \in \sym$ such that
\[
\big(f(x_1 ,  \ldots ,x_n) - f (x_1 ,  \ldots, x_{i-1}, x', x_{i+1}, \ldots ,x_n) \big)^2 \preceq A_i^2
\]
for any $x', x_1 ,  \ldots ,x_n \in \Omega$. In this case,  we instantly get that $ v_f \leq \|\sum_{i=1}^n A_i^2\|$ and hence the inequality 
\begin{equation}\label{ex:concent-PM}
\mu\Big( \lambda_{max}\big(f-\E_\mu f\big)\geq t\Big)\leq 
d\exp\bigg(-\frac{t^2}{ \sigma^2+  t \sigma }\bigg),
\end{equation}
where $\sigma^2:= \| \sum_{i=1}^n A_i^2\|$. This is a weak form of the matrix bounded difference inequality, as Poincar\'e inequality cannot capture sub-Gaussian concentration. The matrix bounded difference inequality with sub-Gaussian tail bounds has been established as a consequence of matrix Azuma inequality \cite[Section~7]{Tropp-con}, and later recovered with improved constant factors as a consequence of a matrix exponential Efron-Stein inequality  \cite[Section~5]{PMT-EffronStein} (see also \cite{PMT} where the inequality is derived with an optimal exponent). While the approach based on the matrix Poincar\'e inequality is unable to compete with such refined results, it provides a unifying framework  for several exponential concentration inequalities allowing to cover a wide range of examples and deriving a variety of concentration inequalities. While Theorem~\ref{th: product measures} is stated for any product measure, it could not be used 
for the example $f(x_1,\ldots,x_n)=\sum_{i=1}^nx_i A_i$ with the standard Gaussian measure as the bounded difference condition is violated in this case.
As a remedy, Theorem~\ref{th: Gaussian-Poincare} provides a refined matrix Poincar\'e inequality yielding the concentration given in Theorem~\ref{thm:Gaussian-prod-Poin}, thus recovering the same bound as in \eqref{ex:concent-PM} for matrix Gaussian series.

Theorem~\ref{th: product measures} follows by combining Theorem~\ref{th: poincare to concentration} with Theorem~\ref{th:Poincare-PM} below, and using  Proposition~\ref{prop:Markov-process-PM} which provides the expression of the matrix carr\'e du champs operator. We start by  introducing a  Markov process $X_t= (X_t^1, \ldots , X_t^n)_{t\in \mathbb{R}_+}$ having $\mu$ as stationary measure and  through which we  obtain a matrix Poincar\'e inequality with respect to a suitable Dirichlet form. Such a construction  is  known, see for instance \cite[Chapter 2]{vanHandel-course}. 

For each coordinate $i=1, \ldots, n $, we associate  an independent Poisson process $N^i=(N_t^i)_{t \in \mathbb{R}_+}$ with  rate $1$ and construct $X_t$ as follows: we draw $X_0$ according to $\mu$ independently of the Poisson process. Then, whenever $N_t^i$ jumps for some $i$, we replace the value of $X_t^i$ by an independent sample from $\mu_i$ while keeping the remaining coordinates fixed.

\begin{proposition}\label{prop:Markov-process-PM}
Let $\mu= \mu_1 \otimes \ldots \otimes \mu_n$ be any product measure on some Polish space $\Omega^n$. The process $(X_t)_{t\in \mathbb{R}_+}$ defined above is a reversible  Markov process with $\mu$ as stationary measure and  semi-group given by 
\[
P_tf(x) = \sum_{I\subseteq \{1, \ldots , n \}} (1-e^{-t})^{|I|} e^{-t(n-|I|)} \int f(x_1, \ldots ,x_n) \prod_{i\in I} d\mu_i(x_i),
\]
for any $x=(x_1, \ldots , x_n)\in \Omega^n$ and any  $f:\Omega^n \to \sym$ whose matrix coefficients belong to $\LL$. 
Moreover, the carr\'e du champ and Dirichlet form are respectively given by
\[
\Gamma (f)(x) = \frac{1}{2} \sum_{i=1}^n \int \big(f(x_1, \ldots, x_n) - f(x_1, \ldots,x_{i-1},z,x_{i+1}, \ldots, x_n)\big)^2 d\mu_i (z) 
\]
and
\[
\Dir (f)= \sum_{i=1}^n \int \bigg(  f(x_1, \ldots, x_n) -\int  f(x_1, \ldots,x_{i-1},z,x_{i+1}, \ldots, x_n)d\mu_i (z)\bigg)^2  d\mu(x).
\]
\end{proposition}
\begin{proof}
It is easy to verify that $X_t$ is a Markov process with $\mu$ as stationary measure and that $X_t$ is reversible with respect to $\mu$. Let $f:\Omega^n \to \sym$ and $x=(x_1, \ldots , x_n)\in \Omega^n$. By  construction, the Markov semi-group is computed explicitly 
\begin{align*}
P_tf(x) 
&=  \sum_{ I\subseteq \{1, \ldots , n \}} \Prob \big[N^i_t>0, \text{ for } i \in I, N^i_t=0 \text{ for } i \notin I \big] \int f(x_1, \ldots ,x_n) \prod_{i\in I} d\mu_i(x_i) 
\\& = \sum_{ I\subseteq \{1, \ldots , n \}} (1-e^{-t})^{|I|} e^{-t(n-|I|)} \int f(x_1, \ldots ,x_n) \prod_{i\in I} d\mu_i(x_i) .
\end{align*}
Moreover as $\lim_{t\rightarrow \infty}(1-e^{-t})^{|I|} e^{-t(n-|I|)}=0$ whenever $|I|<n$,   one can readily see that the process is ergodic. 
In light of Proposition \ref{prop: dirichlet properties}, the carr\'e du champ operator is  given by
\begin{align*}
\Gamma (f)(x)&= \lim_{t\rightarrow 0} \frac{\E\big[\big(f(X_t)-f(X_0)\big)^2\mid X_0=x\big]}{2t} \\
&=\lim_{t\rightarrow 0} \frac{\E\big[\big(f(X_t)-f(x)\big)^2\mid X_0=x\big]}{2t} =\lim_{t\rightarrow 0} \frac{P_t h_{x}(x)}{2t}
\end{align*}
where $h_{x}: \Omega \to \sym$ is the function defined by $h_{x} (y)= (f(x)-f(y))^2$. Now noting that $\lim_{t \rightarrow 0} t^{-1} (1-e^{-t})^{|I|} e^{-t(n-|I|)}=0$ whenever $|I| \geq 2$ and that $h_x(x)=0$,  the  explicit expression of the Markov semigroup then yields that 
\begin{align*}
\lim_{t \rightarrow 0}  \frac{P_t h_x(x)}{t} &= 
\sum_{i=1}^n  \int \big(f(x_1, \ldots, x_n) - f(x_1, \ldots,x_{i-1},z,x_{i+1}, \ldots, x_n)\big)^2 d\mu_i (z).    
\end{align*}
Finally, recalling  that $\Dir (f)= \E_\mu[\Gamma(f)]$ and using that 
\begin{align*}
\int \big(f(x) - &f(x_1, \ldots,x_{i-1},z,x_{i+1}, \ldots, x_n)\big)^2 d\mu_i (z)d\mu_i(x_i)\\
&= 2\int \bigg(  f(x) -\int  f(x_1, \ldots,x_{i-1},z,x_{i+1}, \ldots, x_n)d\mu_i (z)\bigg)^2  d\mu_i(x_i),
\end{align*}
we get the expression of $\Dir(f)$. 
\end{proof}

We are now ready  to prove that  $\mu$ satisfies a matrix Poincar\'e inequality with constant $1$ with respect to the above Dirichlet form. 
\begin{theorem}\label{th:Poincare-PM}
Let $\mu= \mu_1 \otimes \ldots \otimes \mu_n$ be any product measure on some Polish space $\Omega^n$. Then for any   $f:\Omega^n \to \sym$  whose matrix coefficients belong to $\LL$, 
\[
\Var_\mu (f) \preceq \Dir(f).
\]
\end{theorem}
\begin{proof}
Let $f: \Omega^n \to \sym$. Define  $\delta_i f$ by
\[
\delta_i f(x) := f(x) - \int f(x_1, \ldots , x_{i-1}, z ,x_{i+1} , \ldots , x_n) d\mu_i(z)
\] 
and  note that 
\[
\Dir (f)=\sum_{i=1}^n \int  \big(\delta_i f(x)\big)^2 d\mu(x).
\]
Since $(X_t)_{t\in \R_+}$ is ergodic and reversible, we apply Lemma  \ref{lem:var-Dir} to write
\[
\Var_\mu (f)= 2 \int_0^\infty \Dir (P_t f) \, dt =  2  \sum_{i=1}^n  \int_0^\infty \int ( \delta_i P_t f (x))^2 d\mu(x) \, dt. 
\]
Computing  $\delta_i P_tf(x)$ explicitly, we get
\[
\delta_i P_t f (x)  = e^{-t} \sum_{\substack{I\subseteq \{1, \ldots , n \} \\ i \notin I}} (1-e^{-t})^{|I|} e^{-t(n-1-|I|)} \int \delta_i f(x_1, \ldots ,x_n) \prod_{i\in I} d\mu_i(x_i)\, . 
\]
Since $ \sum_{I \not\ni i} (1-e^{-t})^{|I|} e^{-t(n-1-|I|)} =1$ and the square is operator convex, then by convexity and Jensen's inequality we obtain
\begin{align*}
(\delta_i P_t f (x)  )^2 &\preceq e^{-2t} \sum_{\substack{I\subseteq \{1, \ldots , n \} \\ i \notin I}} (1-e^{-t})^{|I|} e^{-t(n-1-|I|)} \Bigg( \int \delta_i f(x_1, \ldots ,x_n) \prod_{i\in I} d\mu_i(x_i)\Bigg)^2\,
\\&
\preceq e^{-2t} \sum_{\substack{I\subseteq \{1, \ldots , n \} \\ i \notin I}} (1-e^{-t})^{|I|} e^{-t(n-1-|I|)}  \int \Big( \delta_i f(x_1, \ldots ,x_n) \Big)^2\prod_{i\in I} d\mu_i(x_i)\,.
\end{align*}
Taking the expectation we get,
\[
\int ( \delta_i P_t f (x))^2 d\mu(x)\preceq e^{-2t} \int ( \delta_i  f (x))^2 d\mu(x)
\]
and hence 
\[
\Var_\mu (f) \preceq  \Bigg(2 \int_0^\infty e^{-2t}  \, dt\Bigg) \sum_{i=1}^n \int ( \delta_i  f (x) )^2 d\mu (x)  = \Dir(f). 
\]
\end{proof}

\begin{remark}\label{rk:EfronStein}
In view of the expression of the Dirichlet form,  the above matrix Poincar\'e inequality implies  the subadditivity of the variance 
\[
\Var_\mu (f) \preceq \Dir(f) =  \sum_{i=1}^n \int \Var_{\mu_i}(f) d\mu,
\]
 and hence the matrix Efron-Stein inequality for product measures. This shows that the latter is a particular case of matrix Poincar\'e inequalities. We refer to \cite[Theorem 5.1]{Ch-Hs} for a direct proof of the Matrix Efron-Stein inequality. 
\end{remark}

\subsection*{Matrix Poincar\'e for the standard Gaussian measure} 
The matrix Poincar\'e inequality established above applies for any product measure. However, when  given a specific product measure, it is possible to architect a suitable Markov generator and prove other matrix Poincar\'e inequalities which could result in better concentration inequalities. 
In the  remaining part of this section, we investigate the case of the $n$-dimensional standard Gaussian measure and prove Theorem~\ref{th: Gaussian-Poincare}.  With this in hand, Theorem~\ref{thm:Gaussian-prod-Poin} will then follow by using Theorem~\ref{th: poincare to concentration} together with the expession of the matrix carr\'e du champ operator given in Proposition~\ref{prop:GaussianDir-CDCh} below.    

As we have seen in Remark \ref{rk:EfronStein}, the matrix Poincar\'e inequality we established can be interpreted as a   matrix Efron Stein inequality. In view of this, it is enough to investigate the matrix Poincar\'e inequality for the one dimensional standard Gaussian measure and then extend it by tensorization to the $n$-dimensional case. To this aim, let us consider the Ornstein-Uhlenbeck semi-group acting on matrix valued functions in the obvious way, by considering the action entrywise. More precisely, the Ornstein-Uhlenbeck semi-group is  defined by
\[
P_t f(x)= \E \big[ f \big( e^{-t}x + \sqrt{1- e^{-2t}}\xi\big)\big], \qquad \xi\sim N(0,1),
\]
for any $f:\R\to \sym$. The Ornstein-Uhlenbeck process is a reversible ergodic Markov process with stationary measure the standard Gaussian measure. Moreover, the associated Markov generator is given by \[\gen (f)(x)= -x f'(x) + f''(x), 
\]
where $f'(x)$ (resp.~$f''(x)$) denotes the matrix whose entries are the derivatives (resp.~ second derivatives) of the corresponding entries of $f(x)$.

\begin{proposition}\label{prop:GaussianDir-CDCh}
The matrix Dirichlet form and matrix carr\'e du champ operator associated with the Ornstein-Uhlenbeck process and standard Gaussian measure $\gamma$ are given by
\[
\Gamma (f) (x) = (f'(x))^2
\quad \text{and} \quad 
\Dir (f) = \E_\gamma[(f')^2], 
\]
for any smooth function $f:\R\to \sym$ whose matrix coefficients and their derivatives belong to $\LL$. 
\end{proposition}
\begin{proof}
To compute the matrix carr\'e du champ operator, we start writing  
\begin{align*}
\Gamma (f)(x)&= \frac12\big( \gen(f^2)(x) -f(x)\gen(f)(x) -\gen(f)(x)f(x)\big)
\\&= \frac12\Big(-x(f^2)'(x) +  ( f^2)''(x) -f(x)\big(-x f'(x) +  f''(x) \big)-\big(-x f'(x) +  f''(x) \big)f(x)\Big)
\\&= (f'(x))^2. 
\end{align*}
Finally, we finish the proof by recalling that the Dirichlet form is the expectation of the carr\'e du champ operator.
\end{proof}

We are ready now to prove the Gaussian  matrix Poincar\'e inequality. 
\begin{proof}[Proof of Theorem~\ref{th: Gaussian-Poincare}]
We first prove the one dimensional version of the theorem. 
Recalling the expression of the semi-group, we note that $(P_t f)'(x) = e^{-t}P_t f' (x)$. Using this together with Property~\eqref{en:square-MSG} of Proposition~\ref{lem: matrix generator prop}, we have  
\begin{align*}
\Dir (P_t f) &= \E_\gamma \big[ \big((P_t f)'\big)^2]
= e^{-2t} \E_\gamma\big[ (P_t  f' )^2\big]
\\&\preceq e^{-2t} \E_\gamma\big[ P_t ( f' )^2\big] = e^{-2t} \E_\gamma\big[  (f' )^2\big] = e^{-2t} \Dir (f). 
\end{align*}
Integrating over $\R_+$ and using Lemma \ref{lem:var-Dir}, we deduce the desired matrix Poincar\'e inequality in the one dimensional case. 
To derive the inequality for the $n$-dimenstional standard Gaussian measure $\mu=\mu_1\otimes\ldots\otimes \mu_n$, we use Remark~\ref{rk:EfronStein} to write 
$$
\Var_\mu (f) \preceq \int \sum_{i=1}^n  \Var_{\mu_i}(f) d\mu,
$$
then use the established one dimensional matrix Poincar\'e to get 
$$
\Var_\mu (f) \preceq  \int \sum_{i=1}^n   \Big(\frac{\partial}{\partial x_i}f(x_1,\ldots,x_n)\Big)^2   d\mu,
$$
and finish the proof.
\end{proof}

\section{Matrix Poincar\'e inequality for SCP measures}\label{sec: strong rayleigh}

The goal of this section is to prove Theorem~\ref{th: strong rayleigh} from which the concentration inequality in Theorem~\ref{th: SR-conc} follows. In the sequel, $\mu$ denotes a probability measure on $\Omega:=\{x\in \{0,1\}^n:\, \sum_{i=1}^n x_i=k\}$ with the SCP property and $\xi$ a random vector on $\Omega$ distributed according to $\mu$. We will start by introducing the (normalized) Markov generator for which $\mu$ satisfies a Poincar\'e inequality with constant $2k$. To this aim, given $x, y\in \Omega$, we denote $x\sim y$ if $x$ and $y$ coincide on all but exactly $2$ coordinates. 

Given $x\sim y$, we denote by $s_{xy}$ (resp.~$s_{yx}$) 
the unique coordinate $i\in \{1,\ldots, n\}$ such that $x_i=0$ and $y_i=1$ (resp.~$x_i=1$ and $y_i=0$). Note that for any two vectors $x,y$ in $\Omega$, it is possible to construct a sequence of intermediate vectors $(z_i)_{0\leq i\leq \ell}$ such that $z_0=x$, $z_\ell=y$ and $z_{i}\sim z_{i+1}$ for any $i=0,\ldots, \ell-1$. Indeed, the intermediate sequence can be derived by swapping zeros and ones (step by step) on the coordinates where $x$ and $y$ differ. This motivates us to build the generator on vectors differing exactly by one such swap. 

Before providing the explicit expression of the generator, let us describe briefly the intuition behind it. 
Given $x\sim y$, to transition from $x$ to $y$, a swap has to be made between the coordinates $s_{xy}$ and $s_{yx}$, and the transition probability is governed by $\mu$. We will uncover  the coordinates of $x$ and $y$ in a random order until reaching the coordinate where the two differ (which could be $s_{xy}$ or $s_{yx}$), in which case we exhibit a ``swapping" probability of this coordinate. The uncovered coordinates will be indexed by an ordered subset $S=(s_1,\ldots,s_\ell)\subset [n]$. We will say that $(S,x,y)$ is \textit{admissible} if $x\sim y$ and $x_S=y_S$, that is, the restriction of $x$ and $y$ to the coordinates in the ordered set $S$ coincide. Note that this automatically implies that $S$ does not contain $s_{xy}$ and $s_{yx}$. 
Now given an admissible triple $(S,x,y)$ and $s\not \in S$, since $\mu$ satisfies the SCP property,  then there exists a coupling $\kappa_{S}^{s}$ of the measures $\mu(\cdot\mid \xi_S=x_S, \xi_s=0)$ and $\mu(\cdot\mid \xi_S=x_S, \xi_s=1)$ which is supported on $\{(x_{\bar{S}}, y_{\bar{S}})\in \{0,1\}^{\bar{S}}\times \{0,1\}^{\bar{S}}:\, x_{\bar{S}}\succeq y_{\bar{S}} \text{ and } \Vert x_{\bar{S}}\Vert_1=k-\Vert x_S\Vert_1\}$, where we denoted by $\bar{S}$ the unordered set $\bar{S}=(S\cup \{s\})^c$. 

We are now ready to introduce the Markov generator $Q$ defined for every $x\sim y$ by 
\begin{equation}\label{eq: def-generator-scp}
Q(x,y):= \frac{1}{2k}\sum_{\ell=0}^{n-2} 
\frac{(n-1-\ell)!}{n!} \sum_{\underset{(S,x,y) \text{ admissible}}{S:\, \vert S\vert=\ell}} 
\frac{H_S^{s_{xy}}(x,y)+H_S^{s_{yx}}(y,x)}{\mu(x\mid \xi_S=x_S)},
\end{equation}
where $$H_S^s(x,y):= \kappa_S^{s}(x,y) \Prob(\xi_{s}=0\mid \xi_S=x_S)\Prob(\xi_{s}=1\mid \xi_S=x_S).$$
We set $Q(x,x)=-\sum_{y\sim x} Q(x,y)$ to complete the construction. 
The above expression puts in place the informal description provided previously. Indeed, it is obtained by averaging over all possible ways of uncovering the coordinates of $x$ and $y$. This can be seen by noting that the factor $\frac{(n-1-\ell)!}{n!}$ represents the probability of  uncovering the coordinates in some fixed order $(s_1,\ldots,s_\ell, s_{xy})$. Finally, after uncovering the coordinates, we exhibit the transition probability on the differing coordinate 
using the corresponding coupling between the measures obtained by conditioning on the uncovered coordinates. We should note that the above Markov generator is the one implicitly used in \cite{H-Salez}. Indeed the above expression can be recovered by carefully following the iterative procedure implemented there. 

Clearly, $Q$ is reversible by construction. Moreover, $Q$ is normalized as we check in the next proposition. 

\begin{proposition}\label{prop: normalisation}
With above notations, we have 
$$
\max_{x\in \Omega}\{-Q(x,x)\} \leq 1.
$$
\end{proposition}

\begin{proof}
Let $x\in \Omega$ and denote by $\supp x :=\{i\in \{1,\ldots, n\}:\, x_i=1\}$ its support. 
We start writing 
$$
-Q(x,x)=\sum_{y\sim x} Q(x,y)= \frac{1}{2k}\sum_{\ell=0}^{n-2} 
\frac{(n-1-\ell)!}{n!} (\alpha_\ell+\beta_\ell),
$$
where 
$$
\alpha_\ell:= \sum_{y\sim x}\, \sum_{\underset{(S,x,y) \text{ admissible}}{S:\, \vert S\vert=\ell}} 
\frac{H_S^{s_{xy}}(x,y)}{\mu(x\mid \xi_S=x_S)},
$$
and 
$$
\beta_\ell:= \sum_{y\sim x}\, \sum_{\underset{(S,x,y) \text{ admissible}}{S:\, \vert S\vert=\ell}} 
\frac{H_S^{s_{yx}}(y,x)}{\mu(x\mid \xi_S=x_S)}.
$$
We will estimate $\alpha_\ell$ and $\beta_\ell$ separately. 

Note that, for $x$ given, the collection of all admissible triples $(S,x,y)$ is in a one to one correspondence with admissible triples $(s,S, y_{\bar{S}})$ where $s\not\in \supp x$, $S\subset [n]\setminus \{s\}$ ordered set, and $y_{\bar{S}}$ a $0/1$ vector on $\bar{S}= (S\cup\{s\})^c$ satisfying $\Vert y_{\bar{S}}\Vert_1= \Vert x_{\bar{S}}\Vert_1-1$. To see this, given $y_{\bar{S}}$, note that one can uniquely define $y\sim x$ by concatenating $x_S$, $y_{\bar{S}}$ and setting $y_s=1$.  
In view of this, we can write 
$$
\alpha_\ell= \sum_{s\not\in \supp x} \, 
\sum_{\underset{S\, \text{ordered},\, \vert S\vert=\ell}{S\subset [n]\setminus \{s\}}}
\sum_{y_{\bar{S}}}
\frac{\kappa_S^{s}(x,y) \Prob(\xi_{s}=0\mid \xi_S=x_S)\Prob(\xi_{s}=1\mid \xi_S=x_S)}{\mu(x\mid \xi_S=x_S)}.
$$
Recalling that 
$$
\sum_{y_{\bar{S}}} \kappa_S^{s}(x,y)= \mu(x\mid \xi_S=x_S, \xi_s=0),
$$
we get 
\begin{align*}
\alpha_\ell&= \sum_{s\not\in \supp x} \, 
\sum_{\underset{S\, \text{ordered},\, \vert S\vert=\ell}{S\subset [n]\setminus \{s\}}} 
 \Prob(\xi_s=1\mid \xi_S=x_S)\\
 &\leq\sum_{\underset{S\, \text{ordered},\, \vert S\vert=\ell}{S\subset [n]}} \sum_{s \in S^c} \E[\xi_s\mid \xi_S=x_S]\\ 
 &= \sum_{\underset{S\, \text{ordered},\, \vert S\vert=\ell}{S\subset [n]}} \sum_{i\not\in S} x_i\\
 &= \sum_{i=1}^n \sum_{\underset{S\, \text{ordered},\, \vert S\vert=\ell}{S\subset [n]\setminus \{i\}}} x_i =   k \frac{(n-1)!}{(n-1-\ell)!}.
\end{align*}
where we used that $ \E[\sum_{s \in S^c}\xi_s\mid \xi_S=x_S]=\sum_{i\not\in S} x_i$ (due to the homogeneity of $\mu$). 

To estimate $\beta_\ell$, note that, given $x$, the collection of all admissible triples $(S,x,y)$ is in a one to one correspondence with all triples $(s,S, y_{\bar{S}})$ where $s\in \supp x$, $S\subset [n]\setminus \{s\}$ ordered set, and $y_{\bar{S}}$ a $0/1$ vector on $\bar{S}= (S\cup\{s\})^c$ satisfying $\Vert y_{\bar{S}}\Vert_1=\Vert x_{\bar{S}}\Vert_1$. To see this, given $y_{\bar{S}}$, note that one can uniquely define $y\sim x$ by concatenating $x_S$, $y_{\bar{S}}$ and setting $y_s=0$.   
Therefore, we have 
$$
\beta_\ell= \sum_{s\in \supp x} \, 
\sum_{\underset{S\, \text{ordered},\, \vert S\vert=\ell}{S\subset [n]\setminus \{s\}}}
\sum_{y_{\bar{S}}}
\frac{\kappa_S^{s}(y,x) \Prob(\xi_{s}=0\mid \xi_S=x_S)\Prob(\xi_{s}=1\mid \xi_S=x_S)}{\mu(x\mid \xi_S=x_S)}.
$$
Recalling that 
$$
\sum_{y_{\bar{S}}} \kappa_S^{s}(y,x)= \mu(x\mid \xi_S=x_S, \xi_s=1),
$$
we get 
$$
\beta_\ell= \sum_{s\in \supp x} \, 
\sum_{\underset{S\, \text{ordered},\, \vert S\vert=\ell}{S\subset [n]\setminus \{s\}}} 
 \Prob(\xi_s=0\mid \xi_S=x_S)\\
 \leq k \frac{(n-1)!}{(n-1-\ell)!}, 
$$
where we used the trivial bound $\Prob(\xi_s=0\mid \xi_S=x_S)\leq 1$. 

Putting together the estimates on $\alpha_\ell$ and $\beta_\ell$, we deduce that 
$$
-Q(x,x)\leq \frac{n-1}{n}\leq 1, 
$$
and finish the proof.
\end{proof}

Before we proceed with the proof of the matrix Poincar\'e inequality promised in the introduction, we need the following identity which we interpret as a two state matrix Poincar\'e. 
\begin{lemma}\label{lem: two-state}
Let $\pi$ be a probability measure on $\{0,1\}$ and $\widetilde Q$ be a reversible Markov generator. Then for any $f: \{0,1\}\to \sym$, we have 
$$
\Var_\pi(f)= \frac{1}{\widetilde Q(0,1)+ \widetilde Q(1,0)} \Dir(f,f).
$$
\end{lemma}
\begin{proof}
First note that 
$$
\Var_\pi(f)= \pi(0)\big(f(0)- \E_\pi[f]\big)^2+ \pi(1)\big(f(1)- \E_pi[f]\big)^2
=\pi(0)\pi(1) \big(f(0)-f(1)\big)^2. 
$$
On the other hand, using the reversibility of $\widetilde Q$, we can write
$$
\Dir(f,f)= \pi(0)\widetilde Q(0,1) \big(f(0)-f(1)\big)^2.
$$
Also, by reversibility of $\widetilde Q$, it is easy to see that  $\pi(1)= \frac{\widetilde Q(0,1)}{\widetilde Q(0,1)+ \widetilde Q(1,0)}$. This finishes the proof. 
\end{proof}

\begin{proof}[Proof of Theorem~\ref{th: strong rayleigh}]
We will show that $\mu$ satisfies a matrix Poincar\'e inequality with constant $2k$ with respect to the Markov generator defined in \eqref{eq: def-generator-scp}. We showed in Proposition~\ref{prop: normalisation} that $Q$ is normalized, and it remains to show that for any $f: \Omega \to \sym$, we have 
$$
\Var_\mu(f) \preceq 2k \Dir(f,f).
$$
First, by Proposition \ref{prop: dirichlet properties}, 


 $$
2k \Dir(f,f) = 
\sum_{\ell=0}^{n-2}\frac{(n-1-\ell)!}{n!} \sum_{x\sim y} 
\sum_{\underset{(S,x,y) \text{ admissible}}{S:\, \vert S\vert=\ell}} 
\Prob(\xi_S=x_S)\, H_S^{s_{xy}}(x,y)\, \big(f(x)-f(y)\big)^2,
$$
where we have used the reversibility of $Q$ to simplify the expression. 
Note that the collection of all admissible triples $(S,x,y)$ 
is in one to one correspondence with all quintuples $(s,S,x_S,x_{\bar{S}},y_{\bar{S}})$,  where $\bar{S}:=(S\cup\{s\})^c$, $s\not\in S$, $x_{\bar{S}}\succeq y_{\bar{S}} \text{ and } \Vert x_{\bar{S}}\Vert_1=k-\Vert x_S\Vert_1$. To see this, note that 
if $(S,x,y)$ is admissible then $x_S=y_S$. Moreover, given 
$(s,S,x_S,x_{\bar{S}},y_{\bar{S}})$, it is possible to uniquely reconstruct $x$ (resp.~$y$) by concatenating $x_S$ and $x_{\bar{S}}$ (resp.~$x_S$ and $y_{\bar{S}}$) and setting $x_s=0$ (resp.~$y_s=1$). In the sequel, given an admissible quintuple $(s,S,x_S,x_{\bar{S}},y_{\bar{S}})$, $x$ and $y$ refer to the vectors constructed as we just described. In view of this, for any $\ell=0,\ldots, n-2$, we can write 
\begin{align*}
    \gamma_\ell&:= \sum_{x\sim y} 
\sum_{\underset{(S,x,y) \text{ admissible}}{S:\, \vert S\vert=\ell}} 
\Prob(\xi_S=x_S)\, H_S^{s_{xy}}(x,y)\, \big(f(x)-f(y)\big)^2\\
&=\sum_{s=1}^n \sum_{\underset{S\, \text{ordered},\, \vert S\vert=\ell}{S\subset [n]\setminus \{s\}}} 
\sum_{x_S} \Prob(\xi_s=0\mid\xi_S=x_S)\,\Prob(\xi_s=1,\, \xi_S=x_S)\sum_{(x_{\bar{S}},y_{\bar{S}})}  \kappa_S^{s}(x,y)\, \big(f(x)-f(y)\big)^2. 
\end{align*}
Since the square is operator convex and $\kappa_S^{s}$ is a probability measure on $\bar{S}\times\bar{S}$, then using Jensen's inequality we get 
$$
\gamma_\ell \succeq 
\sum_{s=1}^n \sum_{\underset{S\, \text{ordered},\, \vert S\vert=\ell}{S\subset [n]\setminus \{s\}}} 
\sum_{x_S} \Prob(\xi_s=0\mid\xi_S=x_S)\,\Prob(\xi_s=1,\, \xi_S=x_S) \big( f_{(s,S,x_S)}(0)-f_{(s,S,x_S)}(1)\big)^2,
$$
where $f_{(s,S,x_S)}: \{0,1\}\to \sym$ is defined by 
$$
f_{(s,S,x_S)}(0)= \sum_{(x_{\bar{S}},y_{\bar{S}})}\kappa_S^{s}(x,y) \, f(x)= 
\sum_{x_{\bar{S}}} \mu(x\mid \xi_s=0, \xi_S=x_S) f(x) = \E[ f(\xi)\mid \xi_s=0, \xi_S=x_S],
$$
and 
$$
f_{(s,S,x_S)}(1)= \sum_{(x_{\bar{S}},y_{\bar{S}})}\kappa_S^{s}(x,y) \, f(y)= 
\sum_{y_{\bar{S}}} \mu(y\mid \xi_s=1, \xi_S=x_S) f(x) = \E[ f(\xi)\mid \xi_s=1, \xi_S=x_S].
$$
Now for a given triple $(s,S,x_S)$, define a probability measure $\pi$ on $\{0,1\}$ by $\pi(0)= \Prob(\xi_s=0\mid\xi_S=x_S)$ and $\pi(1)=\Prob(\xi_s=1\mid\xi_S=x_S)$. Moreover, define a reversible Markov generator $\widetilde Q$ by $\widetilde Q (0,1)= \Prob(\xi_s=1,\, \xi_S=x_S)$. 
On the one hand,  
\begin{align*}
\Dir(f_{(s,S,x_S)},f_{(s,S,x_S)})&=\pi(0) \pi(1) \widetilde{Q}(0,1)  \big( f_{(s,S,x_S)}(0)-f_{(s,S,x_S)}(1)\big)^2\\
&=\Prob(\xi_s=0\mid\xi_S=x_S)\Prob(\xi_s=1,\, \xi_S=x_S) 
 \big( f_{(s,S,x_S)}(0)-f_{(s,S,x_S)}(1)\big)^2.\end{align*}
 On the other hand, by the two-state matrix Poincar\'e inequality (Lemma~\ref{lem: two-state}), 
$$\Dir(f_{(s,S,x_S)},f_{(s,S,x_S)})=\big(\widetilde{Q}(0,1)+\widetilde{Q}(1,0)\big)\Var_\pi(f_{(s,S,x_S)})= \Prob(\xi_S=x_S)\Var_\pi(f_{(s,S,x_S)}).$$
Thus we get 
$$
\Prob(\xi_s=0\mid\xi_S=x_S)\Prob(\xi_s=1,\, \xi_S=x_S) 
 \big( f_{(s,S,x_S)}(0)-f_{(s,S,x_S)}(1)\big)^2
 = \Prob(\xi_S=x_S)\Var_\pi(f_{(s,S,x_S)}), 
$$
which when replaced in the expression of $\gamma_\ell$ yields to 
$$
\gamma_\ell 
\succeq \sum_{s=1}^n \sum_{\underset{S\, \text{ordered},\, \vert S\vert=\ell}{S\subset [n]\setminus \{s\}}} 
\sum_{x_S} \Prob(\xi_S=x_S)\Var_\pi(f_{(s,S,x_S)}).
$$
Now note that 
$$
\E_\pi [f_{(s,S,x_S)}]= \E [f(\xi)\mid \xi_S=x_S],
$$
and 
\begin{align*}
 \sum_{x_S} \Prob(\xi_S=x_S)\E_\pi [f_{(s,S,x_S)}^2]&= 
  \sum_{x_S} \Prob(\xi_s=0, \xi_S=x_S) \big(\E[ f(\xi)\mid \xi_s=0, \xi_S=x_S]\big)^2 \\ 
  & \qquad +  \sum_{x_S} \Prob(\xi_s=1, \xi_S=x_S) \big(\E[ f(\xi)\mid \xi_s=1, \xi_S=x_S]\big)^2\\
&= \sum_{x_{S\cup\{s\}}} \Prob(\xi_{S\cup\{s\}}=x_{S\cup\{s\}}) \big(\E [f(\xi)\mid  \xi_{S\cup\{s\}}=x_{S\cup\{s\}}]\big)^2.  
\end{align*}

Putting together these identities, we get
\begin{align*}
    \gamma_\ell 
&\succeq \sum_{s=1}^n \sum_{\underset{S\, \text{ordered},\, \vert S\vert=\ell}{S\subset [n]\setminus \{s\}}} 
\sum_{x_{S\cup\{s\}}} \Prob(\xi_{S\cup\{s\}}=x_{S\cup\{s\}}) \big(\E [f(\xi)\mid  \xi_{S\cup\{s\}}=x_{S\cup\{s\}}]\big)^2\\
& \qquad\qquad -\sum_{s=1}^n \sum_{\underset{S\, \text{ordered},\, \vert S\vert=\ell}{S\subset [n]\setminus \{s\}}} 
\sum_{x_S} \Prob(\xi_S=x_S) \big(\E [f(\xi)\mid  \xi_{S}=x_{S}]\big)^2\\
&= \sum_{\underset{S\, \text{ordered},\, \vert S\vert=\ell+1}{S\subset [n]}} 
\sum_{x_S} \Prob(\xi_S=x_S) \big(\E [f(\xi)\mid  \xi_{S}=x_{S}]\big)^2 \\
&\qquad \qquad-(n-\ell) \sum_{\underset{S\, \text{ordered},\, \vert S\vert=\ell}{S\subset [n]}} 
\sum_{x_S} \Prob(\xi_S=x_S) \big(\E [f(\xi)\mid  \xi_{S}=x_{S}]\big)^2.
\end{align*}
In view of this, the sum involving $\gamma_\ell$ is a telescopic sum, yielding to 
\begin{align*}
2k\Dir(f,f)&= 
\sum_{\ell=0}^{n-2} \frac{(n-1-\ell)!}{n!} \gamma_\ell \\
&\succeq  \frac{1}{n!}\sum_{\underset{S\, \text{ordered},\, \vert S\vert=n-1}{S\subset [n]}} 
\sum_{x_S} \Prob(\xi_S=x_S) \big(\E [f(\xi)\mid  \xi_{S}=x_{S}]\big)^2  - \big(\E [f(\xi)]\big)^2.   
\end{align*}
Finally, note that by homogeneity, fixing $n-1$ coordinates automatically determines the remaining coordinate. Therefore, 
for any ordered set $S$ of size $n-1$, we have 
$$
\sum_{x_S} \Prob(\xi_S=x_S) \big(\E [f(\xi)\mid  \xi_{S}=x_{S}]\big)^2= \E [f^2(\xi)],
$$
which when replaced in the previous inequality finishes the proof.
\end{proof}

Finally, we end the section by showing how to derive the concentration inequality stated in Theorem~\ref{th: SR-conc}.

\begin{proof}[Proof of Theorem~\ref{th: SR-conc}]
The proof will follow by combining Theorem~\ref{th: poincare to concentration} and Theorem~\ref{th: strong rayleigh}. 
First, using Proposition~\ref{prop: dirichlet properties},  we have for any $f:\Omega\to \sym$ and any $x\in \Omega$
$$
\Gamma(f)(x)= \frac12 \sum_{y\in \Omega} Q(x,y) \big(f(x)-f(y)\big)^2,  
$$
where $Q$ is the Markov generator defined in \eqref{eq: def-generator-scp}. 
Note that if $f$ is $1$-Lipschitz in the sense of Theorem~\ref{th: SR-conc}, then $\Vert f(x)-f(y)\Vert \leq 2$ for any $x\sim y$. Using this together with the triangular inequality, we deduce that for any $1$-Lipschitz matrix function $f$, we have  
$$
\Vert \Gamma(f)(x) \Vert \leq 2, 
$$
for any $x\in \Omega$. 
Replacing this estimate in Theorem~\ref{th: poincare to concentration}  together with the value of the matrix Poincar\'e constant from Theorem~\ref{th: strong rayleigh}, we finish the proof.
\end{proof}

\bigskip
\bigskip

\noindent {\small Richard Aoun,}\\
{\small American University of Beirut, Department of Mathematics, Faculty of Arts and Sciences, P.O. Box 11-0236 Riad El Solh, Beirut 1107 2020, Lebanon,}\\
{\small \it E-mail: ra279@aub.edu.lb}

\bigskip

\noindent {\small Marwa Banna,}\\
{\small Saarland University, Fachbereich Mathematik, 66041 Saarbr\"ucken, Germany}\\
{\small \it E-mail: banna@math.uni-sb.de}

\bigskip

\noindent {\small Pierre Youssef,}\\
{\small Laboratoire de Probabilit\'es, Statistique et Mod\'elisation, 
Universit\'e Paris Diderot, France \\ And \\ Mathematics, Division of Science, New York University Abu Dhabi, UAE}\\
{\small \it E-mail: youssef@lpsm.paris}


\nocite{*}
 \begin{thebibliography}{99}
\bibitem{AW}
R.~Ahlswede and A.~Winter.
\newblock Addendum to: {S}trong converse for identification via quantum
  channels. 
\newblock {\em IEEE Trans. Inform. Theory}, 49(1):346, 2003.

\bibitem{AS}
\newblock  S.~Aida, D.~Stroock. 
\newblock  Moment estimates derived from Poincar\'e and logarithmic Sobolev inequalities. 
\newblock {\em Math. Res. Lett.}1, 75--86 (1994).

\bibitem{BMY-betamixing}
M.~Banna, F.~Merlev\`ede, and P.~Youssef.
\newblock Bernstein-type inequality for a class of dependent random matrices.
\newblock {\em Random Matrices Theory Appl.}, 5(2):1650006, 28, 2016.

\bibitem{BBL}
J.~Borcea, P.~Br\"and\'en, and T.~M.~Liggett.
\newblock Negative dependence and the geometry of polynomials.
\newblock {\em J. Amer. Math. Soc.}, 22(2):521--567, 2009.

\bibitem{BLM-book}
S.~Boucheron, G.~Lugosi, and P.~Massart.
\newblock {\em Concentration inequalities}.
\newblock Oxford University Press, Oxford, 2013.
\newblock A nonasymptotic theory of independence, With a foreword by Michel
  Ledoux.
  
  
\bibitem{Carlen}
E.~Carlen.
\newblock Trace inequalities and quantum entropy: an introductory course.
\newblock In {\em Entropy and the quantum}, volume 529 of {\em Contemp. Math.},
  pages 73--140. Amer. Math. Soc., Providence, RI, 2010.
  
 
\bibitem{Chen-Tropp}
R.~Y. Chen and J.~A. Tropp.
\newblock Subadditivity of matrix {$\phi$}-entropy and concentration of random
  matrices.
\newblock {\em Electron. J. Probab.}, 19:no. 27, 30, 2014.



\bibitem{Ch-Hs}
H.-C. Cheng and M.-H. Hsieh.
\newblock Characterizations of matrix and operator-valued {$\Phi$}-entropies,
  and operator {E}fron-{S}tein inequalities.
\newblock {\em Proc. R. Soc. A}, 472(2187):20150563, 20, 2016.

\bibitem{Ch-Hs-new}
H.-C. Cheng and M.-H. Hsieh.
\newblock Matrix Poincaré, $\Phi$-Sobolev inequalities, and quantum ensembles.
\newblock {\em J. Math. Phys.}, 60 (3):032201, 2019.

 
\bibitem{Ch-Hs-17}
H.-C. Cheng, M.-H. Hsieh, and M.~Tomamichel.
\newblock Exponential decay of matrix {$\Phi$}-entropies on {M}arkov semigroups
  with applications to dynamical evolutions of quantum ensembles.
\newblock {\em J. Math. Phys.}, 58(9):092202, 24, 2017.


\bibitem{Choi}
M.D. Choi.
\newblock A Schwarz inequality for positive linear maps on $C^*$-algebras.
\newblock {\em Illinois J. Math.}, 18 (1974), pp. 565--574.


\bibitem{F-Th}
P.~J. Forrester and C.~J. Thompson.
\newblock The {G}olden-{T}hompson inequality: historical aspects and random
  matrix applications.
\newblock {\em J. Math. Phys.}, 55(2):023503, 12, 2014.


\bibitem{Golden}
S.~Golden.
\newblock Lower bounds for the {H}elmholtz function.
\newblock {\em Phys. Rev. (2)}, 137:B1127--B1128, 1965.


\bibitem{Gro-Mil}
M.~Gromov and V.~D. Milman.
\newblock A topological application of the isoperimetric inequality.
\newblock {\em Amer. J. Math.}, 105(4):843--854, 1983.


\bibitem{Hansen-Pedersen}
F.~Hansen and  G.~K.~Pedersen.
\newblock Jensen's inequality for operators and L\"{o}wnerÕs theorem. 
\newblock {\em Mathematische Annalen}, 258(1982), 229--241.

\bibitem{H-Salez}
J.~Hermon and J.~Salez.
\newblock Modified log-sobolev inequalities for strong-rayleigh measures.
\newblock {\em arXiv preprint arXiv:1902.02775}, 2019.

 \bibitem{Hiai-Kosaki}
F.~Hiai and H.~Kosaki.
\newblock Means for matrices and comparison of their norms. 
\newblock {\em Indiana Univ. Math. J.}, 48 (1999), 899--936.

\bibitem{KS}
R.~Kyng and Z.~Song.
\newblock A matrix chernoff bound for strongly rayleigh distributions and spectral sparsifiers from a fewrandom spanning trees.
\newblock In 2018 IEEE 59th Annual Symposium on Foundations of Computer Science(FOCS).  IEEE, 2018, pp. 373--384.

\bibitem{Ledoux-lecturenotes}
M.~Ledoux.
\newblock Concentration of measure and logarithmic {S}obolev inequalities.
\newblock In {\em S\'{e}minaire de {P}robabilit\'{e}s, {XXXIII}}, volume 1709
  of {\em Lecture Notes in Math.}, pages 120--216. Springer, Berlin, 1999.

\bibitem{ledoux}
M.~Ledoux.
\newblock The concentration of measure phenomenon.
\newblock Mathematical Surveys and Monographs,  vol.  89.   American  Mathematical  Society,  Providence,  RI(2001).



\bibitem{yau}
S.~L.~Lu and H.-T.~Yau.  
\newblock Spectral gap and logarithmic Sobolev inequality for Kawasaki and Glauber dynamics. 
\newblock {\em Comm. Math. Phys.}, 156(2):399--433, 1993. 


\bibitem{MatrixCon-exch-pairs}
L.~Mackey, M.~I. Jordan, R.~Y. Chen, B.~Farrell, and J.~A. Tropp.
\newblock Matrix concentration inequalities via the method of exchangeable
  pairs.
\newblock {\em Ann. Probab.}, 42(3):906--945, 2014.



\bibitem{Oliveira}
R.~I. Oliveira.
\newblock Sums of random {H}ermitian matrices and an inequality by {R}udelson.
\newblock {\em Electron. Commun. Probab.}, 15:203--212, 2010.

\bibitem{PMT}
D.~Paulin, L.~Mackey, and J.~A. Tropp.
\newblock  Deriving matrix concentration inequalities from kernel couplings.
\newblock Available at arXiv:1305.0612, 2014.

\bibitem{PMT-EffronStein}
D.~Paulin, L.~Mackey, and J.~A. Tropp.
\newblock Efron-{S}tein inequalities for random matrices.
\newblock {\em Ann. Probab.}, 44(5):3431--3473, 2016.

\bibitem{Pemantle-Peres}
R.~Pemantle and Y.~Peres.
\newblock Concentration of {L}ipschitz functionals of determinantal and other
  strong {R}ayleigh measures.
\newblock {\em Combin. Probab. Comput.}, 23(1):140--160, 2014.

\bibitem{Thompson}
C.~J. Thompson.
\newblock Inequality with applications in statistical mechanics.
\newblock {\em J. Mathematical Phys.}, 6:1812--1813, 1965.


\bibitem{Tropp-martingale}
J.~A. Tropp.
\newblock Freedman's inequality for matrix martingales. 
\newblock{\em Electronic Communications in Probability}, 16:   262--270, 2011.
 


\bibitem{Tropp-con}
J.~A. Tropp.
\newblock User-friendly tail bounds for sums of random matrices.
\newblock {\em Found. Comput. Math.}, 12(4):389--434, 2012.

\bibitem{Tropp-survey}
J.~A. Tropp.
\newblock An introduction to matrix concentration inequalities.
\newblock {\em Foundations and Trends{\textregistered} in Machine Learning},
  8(1-2):1--230, 2015.


  
\bibitem{vanHandel-course}
R.~van Handel.
\newblock Probability in high dimension. 
\newblock ORF 570 Lecture Notes, Princeton University, June 2014.
\end{thebibliography}
\end{document}